\documentclass{article}
\usepackage[a4paper, left=1.8cm, right=1.8cm, top=2.5cm, bottom=1.5cm]{geometry}
\usepackage{authblk}
\usepackage{graphicx} 
\usepackage{amsmath, amssymb, amsthm}
\allowdisplaybreaks[4]
\usepackage{graphicx, epstopdf, epsfig, multirow, epic, bm}
\usepackage[raggedright]{titlesec}
\usepackage{hyperref} 
\usepackage{appendix}  
\usepackage{mathrsfs} 
\usepackage{algorithm}
\usepackage{algpseudocode}   
\usepackage{array}
\usepackage{authblk}
\usepackage{float}
\usepackage{booktabs} 
\usepackage{multirow}
\usepackage{hyperref} 
\hypersetup{
	colorlinks=true,
	linkcolor=blue,
	urlcolor=blue,
	citecolor=blue,
}
\usepackage{cleveref} 
\usepackage{colortbl}

\usepackage{caption}
\usepackage{subfigure}

\newtheorem{theorem}{Theorem}[section]
\newtheorem{remark}{Remark}[section]
\newtheorem{lemma}{Lemma}[section]
\newtheorem{assumption}{Assumption}[section]

\newcommand{\be}{\begin{equation}}
	\newcommand{\ee}{\end{equation}}
\newcommand{\ba}{\begin{array}}
	\newcommand{\ea}{\end{array}}
\newcommand{\bea}{\begin{eqnarray}}
	\newcommand{\eea}{\end{eqnarray}}

\def\BibTeX{{\rm B\kern-.05em{\sc i\kern-.025em b}\kern-.08em
		T\kern-.1667em\lower.7ex\hbox{E}\kern-.125emX}}

\usepackage[table,xcdraw]{xcolor}
\usepackage{listings} 
\title{A Constant-Modulus Optimization Approach to Quadratic $K$-Medoids Graph Clustering}
\author[1]{Yutong Zheng}
\author[1,2]{Qingna Li}
\author[1,*]{Wenshun Teng}
\affil[1]{School of Mathematics and Statistics, Beijing Institute of Technology, Beijing, China}
\affil[2]{Beijing Key Laboratory on MCAACI/Key Laboratory of Mathematical Theory and Computation in Information Security, Beijing Institute of Technology, Beijing, China}
\affil[*]{Corresponding author}
\date{October 1, 2026}

\begin{document}
	
	\maketitle
	
	\begin{abstract}
		Graph clustering aims to partition the vertices of a graph into groups with dense intra-cluster connections and sparse inter-cluster connections. In this paper, we study a quadratic $K$-medoids formulation for graph clustering from the perspective of CM optimization. By identifying the binary feasible set as a discrete CM set, we establish a connection between the quadratic $K$-medoids model and CM optimization. Rather than solving the original binary problem directly, we consider the convex hull of its feasible set and introduce a negative squared-norm penalty to promote extreme-point solutions, resulting in an extreme point pursuit model over a simple convex set. The efficient projection onto this convex set enables the use of projected-gradient-based optimization for large-scale medoid selection. Based on this formulation, we propose a two-stage clustering method. In Stage~1, a penalty-continuation projected gradient solver is used to approximately solve the penalized medoid-selection problem. In Stage~2, the selected medoids are converted into a graph partition through a unified two-level assignment rule, where Jaccard distance is used as the primary criterion and shortest-path distance is used as a secondary criterion when needed. Numerical experiments on synthetic PPM and SBM graphs, together with real-world graph data, demonstrate the effectiveness and computational efficiency of our method. The experiments also show that a smaller objective value in the medoid-selection problem does not necessarily lead to a more accurate final graph partition, highlighting the importance of evaluating both optimization performance and clustering quality after assignment.
	\end{abstract}

    \textbf{Keywords: Graph clustering; Network community detection; $K$-medoids; Binary quadratic optimization; Constant-modulus problems; Projected gradient method  }
	\section{Introduction}
	
	Graph clustering, also known as community detection, is a fundamental task in network analysis, and its solution methodologies span a wide range of disciplines, including statistical inference, dynamical systems, and optimization \cite{schaeffer2007graph,fortunato2010community,fortunato2016community}. Among them, optimization-based approaches have attracted considerable attention because of their relatively low computational cost and high accuracy, particularly methods based on modularity optimization \cite{intro-24,intro-add-24-1}. Although such methods offer several advantages, including short computational time and no requirement to pre-specify the number of clusters \cite{miasnikof2024graph}, they also suffer from several well-recognized theoretical and practical limitations, such as restrictions on the detectable clustering structures \cite{intro-70,aref2023heuristic}, the resolution limit \cite{intro-71,cq-2}, and degeneracy issues \cite{intro-71,cq-1}. These limitations have motivated extensive research into alternative approaches that do not rely on modularity maximization. 
	In 2010, Fan and Pardalos \cite{fan2010linear} pioneered a general graph partitioning model that optimizes graph partitioning by maximizing the total weight of intra-subset edges for a prespecified number of subsets. In 2012, Fan et al. \cite{intro-72} introduced an interval-uncertainty framework and developed a robust optimization model, thereby extending graph clustering to uncertain settings. In 2020, Miasnikof et al. \cite{gc-no34} proposed a combinatorial optimization model based on maximizing intra-cluster density. More recently, in 2024, Miasnikof et al. \cite{miasnikof2024graph} applied the QUBO-based $K$-medoids clustering model introduced by Bauckhage et al. \cite{k-medoids} in 2019 to graph clustering for the first time. They further developed a tailored Boltzmann-machine heuristic combined with parallel-tempering simulated annealing to reduce the risk of becoming trapped in poor local optima.
	
	This $K$-medoids model is particularly well suited for graph clustering because the selected representative vertices are actual vertices of the graph and therefore admit a direct interpretation. Moreover, it relies only on pairwise distance information and does not require an embedding in Euclidean space. When combined with structure-aware distance measures, such as the Jaccard distance, the model can effectively identify representative vertices from highly connected communities. However, the trade-off parameters $\alpha$ and $\beta$ need to be determined through empirical tuning, and the model is more sensitive to perturbations, resulting in greater performance variability in high-noise settings. Nevertheless, in most test scenarios, it can achieve more accurate clustering results within a short computational time. The associated quadratic formulation further balances centrality and separation, encouraging the selected medoids to be both representative of their communities and sufficiently distinct from one another.  These considerations provide the main motivation for the research presented in this paper.

	The key observation of this paper is that the feasible set of the quadratic $K$-medoids model has a natural constant-modulus (CM) structure \cite{liu2024expp1,liu2024expp2}. Indeed, the medoid-selection variable is a binary vector with exactly $K$ nonzero components, and hence every feasible point satisfies $\|{\bf x}\|_2^2=K$. Therefore, all feasible points lie on a sphere with a prescribed radius. This fixed-norm structure shows that the quadratic $K$-medoids problem naturally belongs to the class of CM optimization problems. Nevertheless, this geometric structure has not yet been explicitly exploited in the context of quadratic $K$-medoids graph clustering. In this paper, we focus on medoid selection and partition construction for undirected, unweighted graphs without self-loops or multiple edges, with vertices assigned to non-overlapping clusters. Accordingly, we do not provide a further review of the literature on CM optimization problems.

	Building on the CM structure of the binary feasible set, we consider the convex hull of the feasible set of the original quadratic $K$-medoids graph clustering model, rather than solving the binary formulation directly. By introducing a negative squared-norm penalty to promote extreme-point solutions, we obtain an extreme point pursuit model formulated as a continuous optimization problem over a simple convex set. The projection onto this convex set can be computed efficiently, providing a favorable foundation for applying projected gradient methods to medoid selection in large-scale graph clustering problems. On this basis, we propose a two-stage clustering method (TSCM) that combines continuous optimization-driven medoid selection with a unified graph-distance-based assignment procedure, thereby transforming the selected medoids into a final graph partition. The main contributions of this paper are summarized as follows:
	\begin{itemize} 
		\item We identify the binary feasible set in the quadratic $K$-medoids graph clustering model as a discrete CM set, thereby connecting the model with CM optimization. 
		\item We reformulate the quadratic $K$-medoids problem over the convex hull of the binary feasible set and introduce an extreme-point-promoting negative squared-norm penalty. The resulting model is a continuous optimization problem over the convex hull of the feasible set. 
		\item We develop a two-stage clustering method. In Stage~1, a penalty-continuation projected gradient solver is designed to approximately solve the penalized medoid-selection problemthrough refined multi-start initialization, a progressively increasing penalty parameter, and a BB projected gradient method combined with a projected Armijo line search. Binary candidate solutions are retained along the continuation path.
		In Stage~2, the selected medoids are converted into	a graph partition through a unified assignment procedure, which employs the Jaccard distance as the primary criterion and the shortest-path distance as the secondary criterion.
		\item We conduct numerical experiments on synthetic and real-world graphs to evaluate the clustering effectiveness and computational efficiency of TSCM. The results further reveal that a smaller objective value in the medoid-selection problem does not necessarily lead to a more accurate final graph partition, highlighting the importance of evaluating both optimization performance and final clustering quality.
		
	\end{itemize}

	\section{Preliminaries}
	\label{sec:preliminaries}
	For ease of reading, this section introduces a class of constant modulus problems and a quadratic $K$-medoids formulation for graph clustering.
	\subsection{A Class of Constant Modulus Problems}
	\label{subsec:cm-problems}
	
	We first recall a class of CM optimization problems.
	Let $\mathcal V\subseteq\mathbb R^n$ be a nonempty closed
	set. We say that $\mathcal V$ is a constant modulus set with modulus
	$\sqrt{C}$ if
	\[
	\|{\bf x}\|_2^2=C,
	\qquad
	\forall {\bf x}\in\mathcal V.
	\]
	A generic CM optimization problem takes the form
	\begin{equation}\label{eq:general-cm}
		\min_{{\bf x}\in\mathcal V} f({\bf x}),
	\end{equation}
	where $f:\mathcal D\to\mathbb R$ is an objective function defined on a
	domain $\mathcal D\subseteq\mathbb R^n$, and
	\[
	\mathcal V\subseteq \mathcal D,
	\qquad
	\operatorname{conv}(\mathcal V)\subseteq \mathcal D.
	\]
	The condition $\|{\bf x}\|_2^2=C$ means that all feasible points have the
	same Euclidean norm. Hence, although CM sets may have very different
	algebraic forms, they share a common geometric feature: they lie on a
	sphere of radius $\sqrt{C}$. This feature is central to the construction
	of the penalized formulations considered in this paper.
	
	Several important nonconvex constraint sets can be viewed as CM sets. Typical
	examples are summarized in Table~\ref{tab:cm-examples}.
	\begin{table}[H]
		\centering
		\caption{Examples of constant-modulus sets.}
		\label{tab:cm-examples}
		\begin{tabular}{lll}
			\hline
			Set & Definition & Constant squared norm $C$ \\
			\hline
			Binary vectors
			&
			$\{-1,1\}^n$
			&\quad\quad\quad\quad
			$n$
			\\[2mm]
			
			Unit sphere
			&
			$\mathbb S^{n-1}:=\{{\bf x}\in\mathbb R^n\mid \|{\bf x}\|_2=1\}$
			&\quad\quad\quad\quad
			$1$
			\\[2mm]
			
			Stiefel manifold
			&
			$\mathrm{St}(n,r):=\{X\in\mathbb R^{n\times r}\mid X^\top X=I_r\}$
			&\quad\quad\quad\quad
			$r$
			\\[2mm]
			
			Unit vectors
			&
			$\{{\bf e}_1,\ldots,{\bf e}_n\}\subseteq\mathbb R^n$
			&\quad\quad\quad\quad
			$1$
			\\
			\hline
		\end{tabular}
	\end{table}

	
	The main difficulty in solving \eqref{eq:general-cm} lies in the
	nonconvexity of $\mathcal V$. A natural idea is to replace $\mathcal V$
	by its convex hull $\operatorname{conv}(\mathcal V)$. However, the direct
	convex-hull relaxation
$
	\min_{{\bf x}\in\operatorname{conv}(\mathcal V)} f({\bf x}) 
$
	is generally not equivalent to the original CM problem, since the relaxed
	problem may admit fractional or non-CM minimizers. 
	
	To recover the original
	CM structure, the extreme point pursuit (EXPP) framework \cite{liu2024expp1} introduces the
	penalized formulation
	\begin{equation}\label{eq:general-expp}
		\min_{{\bf x}\in\operatorname{conv}(\mathcal V)}
		F_\lambda({\bf x})
		:=
		f({\bf x})-\lambda\|{\bf x}\|_2^2,
		\qquad
		\lambda>0.
	\end{equation}
The negative squared-norm term promotes large-norm points over the convex hull
and hence encourages solutions to move toward the original CM set. In
Section~\ref{sec:selection-expp}, we specialize this idea to the
fixed-cardinality feasible set arising from the quadratic $K$-medoids problem.

	\subsection{A Quadratic $K$-medoids Formulation for Graph Clustering}
	\label{subsec:quadratic-k-medoids}
	
	We now introduce the quadratic $K$-medoids formulation for graph clustering used in this paper. 
	
	Let $G=(V,E)$ be an undirected graph with $n=|V|$ vertices, and let
	$
	D=(d_{ij})\in\mathbb R^{n\times n} 
	$
	be a vertex--vertex distance matrix, where $d_{ij}$ measures the dissimilarity
	between vertices $i$ and $j$. The construction of $D$ will be specified in
	Subsection~\ref{subsubsec:jaccard-distance}.
	
	The $K$-medoids model selects $K$ representative vertices, called medoids, from
	the vertex set itself. Unlike centroids in $K$-means, medoids are actual
	vertices of the graph and are therefore directly interpretable. Let
	\[
	{\bf x}=(x_1,\ldots,x_n)^\top\in\{0,1\}^n
	\]
	be the medoid-selection vector, where
	\[
	x_i=
	\begin{cases}
		1, & \text{if vertex } i \text{ is selected as a medoid},\\
		0, & \text{otherwise}.
	\end{cases}
	\]
	The requirement that exactly $K$ medoids are selected is expressed as ${\bf 1}^\top{\bf x}=K$.
	
	Following the quadratic $K$-medoids formulation, we consider two competing criteria. The first one is a centrality term,
\[
{\bf x}^\top D{\bf 1}
=
\sum_{i=1}^n x_i\sum_{j=1}^n d_{ij},
\]
which measures the total distance from the selected medoids to all vertices. Minimizing this term encourages the selected medoids to be globally central. The second one is a scattering term,
\[
\frac{1}{2}{\bf x}^\top D{\bf x}
=
\sum_{1\le i<j\le n} x_i x_j d_{ij},
\]
which measures the total pairwise distance among selected medoids. Maximizing this term encourages the medoids to be well separated and avoids selecting redundant representatives from the same region of the graph.
	
	Balancing these two effects leads to the following quadratic binary
	optimization problem:
	\begin{equation}\label{eq:kmedoids-binary}
		\begin{aligned}
			\min_{{\bf x}}\quad
			& f({\bf x})
			:=
			\beta{\bf x}^{\top}D{\bf 1}
			-
			\frac{\alpha}{2}{\bf x}^{\top}D{\bf x}\\
			\text{s.t.}\quad
			& {\bf 1}^{\top}{\bf x}=K,\\
			& x_i\in\{0,1\},\quad i=1,\ldots,n,
		\end{aligned}
	\end{equation}
	where $\alpha>0$ and $\beta>0$ are trade-off parameters. The parameter $\beta$
	controls the importance of selecting central vertices, while $\alpha$ controls
	the importance of selecting mutually distant medoids.
	
	Problem \eqref{eq:kmedoids-binary} provides a compact quadratic	formulation for graph clustering. It selects $K$ representative vertices by trading off centrality and scattering, uses a graph-based distance matrix to capture structural dissimilarity. 
	Its feasible set has CM structure, which 
	be exploited in the next section to
	derive an exact-penalty continuous reformulation over its convex hull.

	\subsection{Distance Selection: Jaccard distance}   
	\label{subsubsec:jaccard-distance}
	
	Graph clustering aims to partition vertices into groups whose members are
	similar in terms of their connectivity patterns. In this setting, the
	notion of distance should reflect structural similarity rather than merely
	the length of a shortest path. For this reason, we adopt the Jaccard
	distance, which measures the dissimilarity between two vertices according
	to the overlap of their neighborhoods.
	
	Let $G=(V,E)$ be an undirected graph. For each vertex $i\in V$, denote by
	\[
	a_i:=\{v\in V\mid (i,v)\in E\}
	\]
	the neighborhood of vertex $i$. The Jaccard similarity between vertices
	$i$ and $j$ is defined as
	\[
	s_{ij}
	:=
	\frac{|a_i\cap a_j|}{|a_i\cup a_j|},
	\]
	and the corresponding Jaccard distance is
	\begin{equation}\label{eq:jaccard-distance}
		d_{ij}
		:=
		1-s_{ij}
		=
		1-
		\frac{|a_i\cap a_j|}{|a_i\cup a_j|}
		\in[0,1].
	\end{equation}
	The Jaccard distance was
	originally introduced in botany \cite{jaccard-1} and has since been used
	in many areas involving set-based similarity comparisons.
	
	The benefits of using the Jaccard distance can be mainly summarized as follows: (1) Wide Application. Jaccard distance is widely used in various fields, such as complex networks \cite{jaccard-2}, text classification \cite{jaccard-5} and plant recognition \cite{jaccard-6}. (2) Simple Computation: All calculations are pairwise, which can be parallelized or incrementally processed. They do not require storing the entire adjacency matrix or Laplacian matrix in memory, nor do they require any expensive computations like inversion or eigendecomposition. (3) Clear Interpretability. Jaccard distance can capture the similarity of connectivity between vertices and community structures, and it can also capture the discreteness of the graph. 
	
	In the context of graph clustering, previous studies
	\cite{jaccard-3,jaccard-4} have also demonstrated that Jaccard-type
	distances can provide effective measures of vertex dissimilarity. These
	properties make the Jaccard distance a natural choice for constructing the
	distance matrix $D=(d_{ij})$ used in this paper. 

	\section{Constant-Modulus Exact Reformulation of the Quadratic $K$-medoids Problems}
	\label{sec:selection-expp}
	
	We now show that the quadratic $K$-medoids formulation can be viewed as a special case of the constant modulus (CM) problem. This observation allows us to replace the discrete feasible set by its convex hull and to use the negative squared-norm penalty proposed in the EXPP framework. The role of this section is therefore twofold. Firstly,
	we identify the precise CM structure of our model. Secondly, we prove that, under a standard Lipschitz-gradient condition, the resulting penalized convex-hull formulation is an exact reformulation of the original problem.
	
	Let ${\bf x}=(x_1,x_2,\ldots,x_n)^\top\in\mathbb R^n$. Consider the problem \eqref{eq:kmedoids-binary}.
	The feasible set of \eqref{eq:kmedoids-binary} is the selection vector set
	\begin{equation}\label{eq:selection-set}
		\mathcal U_k^n
		=
		\left\{
		{\bf x}\in\{0,1\}^n
		\mid
		{\bf 1}^{\top}{\bf x}=K
		\right\}.
	\end{equation}
	For every ${\bf x}\in\mathcal U_k^n$, exactly $K$ components of ${\bf x}$ are equal to one and the remaining components are equal to zero. Hence
	\[
	\|{\bf x}\|_2^2
	=
	\sum_{i=1}^n x_i^2
	=
	\sum_{i=1}^n x_i
	=
	K,
	\qquad
	\forall {\bf x}\in\mathcal U_k^n .
	\]
	Therefore, $\mathcal U_k^n$ is a CM set with modulus $\sqrt{K}$. Moreover, its convex hull is given by
	\[
	\operatorname{conv}(\mathcal U_k^n)
	=
	C_k
	:=
	\left\{
	{\bf x}\in\mathbb R^n
	\mid
	{\bf 1}^{\top}{\bf x}=K,\ 
	0\le x_i\le 1,\ i=1,\ldots,n
	\right\}.
	\]
	
	We next recall the basic concave-minimization principle underlying the EXPP reformulation. The principle says that minimizing a strictly concave function over the convex hull of a set does not create new minimizers outside the original set. This is the key mechanism by which the convex hull relaxation becomes exact.
	
	\begin{lemma}\cite[Lemma 1]{liu2024expp1}
		\label{lem:concave-min}
		Let $\mathcal A\subseteq \mathbb R^n$ be a nonempty set, and let $\Phi:\operatorname{conv}(\mathcal A)\to\mathbb R$ be strictly concave. Assume that
		\[
		\min_{{\bf x}\in \mathcal A}\Phi({\bf x})
		\]
		has an optimal solution. Then
		\[
		\min_{{\bf x}\in \operatorname{conv}(\mathcal A)}
		\Phi({\bf x})
		\]
		is an exact reformulation of the above problem in the following sense:
		the two problems have the same optimal solution set.
	\end{lemma}
		
	Lemma~\ref{lem:concave-min} suggests that the direct relaxation	$\min_{{\bf x}\in C_k}f({\bf x})$ would be exact if $f$ itself were	strictly concave on $C_k$. In general, however, this strict concavity cannot be guaranteed. This motivates the EXPP penalty $-\lambda\|{\bf x}\|_2^2$, whose role is to concavify the objective and drive the minimizer toward an extreme point of $C_k$.
	
	We impose the following standard smoothness assumption.
	
	\begin{assumption}
		\label{ass:lipschitz-gradient-selection}
		The function $f$ has $L$-Lipschitz continuous gradient on $C_k$, i.e., there exists a constant $L>0$ such that
		\[
		\|\nabla f({\bf x})-\nabla f({\bf y})\|_2
		\le
		L\|{\bf x}-{\bf y}\|_2,
		\qquad
		\forall {\bf x},{\bf y}\in C_k.
		\]
	\end{assumption}
	
	\begin{remark}
		Assumption~\ref{ass:lipschitz-gradient-selection} is mild in the present setting. Since $f$ is a quadratic function, its gradient is Lipschitz continuous on the whole space. Indeed, letting
		\[
		S:=\frac{D+D^\top}{2},
		\]
		we have ${\bf x}^{\top}D{\bf x}={\bf x}^{\top}S{\bf x}$, and hence $\nabla f({\bf x})=\beta D{\bf 1}-\alpha S{\bf x}$. 	Thus, one may take $L=\alpha\|S\|_2$ when $\alpha\ge 0$. If $D$ is symmetric, then $S=D$ and $L=\alpha\|D\|_2$.
	\end{remark}
	
	The next lemma explains why the negative squared-norm term makes the objective strictly concave when the penalty parameter is sufficiently large. 
	
	\begin{lemma}\cite[Lemma 2]{liu2024expp1}
		\label{lem:concavification-selection}
		Suppose that Assumption~\ref{ass:lipschitz-gradient-selection} holds. Then, for any $\lambda>L/2$, the function
		\[
		g_\lambda({\bf x})
		:=
		f({\bf x})-\lambda\|{\bf x}\|_2^2
		\]
		is strictly concave on $C_k$.
	\end{lemma}
	
	\begin{proof}
		Since $f$ has $L$-Lipschitz continuous gradient on $C_k$, the standard descent lemma implies that $-f$ is $L$-weakly convex on $C_k$. Equivalently, for any $\mu>L$, the function
		\[
		f({\bf x})-\frac{\mu}{2}\|{\bf x}\|_2^2
		\]
		is strictly concave on $C_k$. Taking $\mu=2\lambda$, we obtain strict concavity of
		\[
		f({\bf x})-\lambda\|{\bf x}\|_2^2
		\]
		whenever $2\lambda>L$, or equivalently $\lambda>L/2$.
	\end{proof}
	
	We can now state the exact reformulation result for our model. 
	
	\begin{theorem}
		\label{thm:selection-expp-exact}
		Suppose that Assumption~\ref{ass:lipschitz-gradient-selection} holds. For any $\lambda>L/2$, the problem \eqref{eq:kmedoids-binary} is equivalent to the following convex-hull penalized problem:
		\begin{equation}\label{k-meroid-expp}
			\begin{aligned}
				\min_{{\bf x}}\quad
				& g_\lambda({\bf x})
				:=
				f({\bf x})-\lambda\|{\bf x}\|_2^2\\
				\text{s.t.}\quad
				& {\bf x}^{\top}{\bf 1}=K,\\
				& 0\leq x_i\leq 1,\quad i=1,\ldots,n .
			\end{aligned}
		\end{equation}
		More precisely, the optimal solution set of \eqref{k-meroid-expp} is equal to the optimal solution set of \eqref{eq:kmedoids-binary}. The optimal values differ only by the constant $-\lambda K$.
	\end{theorem}
	
	\begin{proof}
		The feasible set of	\eqref{k-meroid-expp} is exactly
		$
		C_k=\operatorname{conv}(\mathcal U_k^n). 
		$
		By Lemma~\ref{lem:concavification-selection}, $g_\lambda$ is strictly concave on $C_k$ whenever $\lambda>L/2$. Applying
		Lemma~\ref{lem:concave-min} with
		\[
		\mathcal A=\mathcal U_k^n,
		\qquad
		\Phi=g_\lambda,
		\]
		we obtain that
		\[
		\min_{{\bf x}\in C_k}g_\lambda({\bf x})
		\]
		and
		\[
		\min_{{\bf x}\in\mathcal U_k^n}g_\lambda({\bf x})
		\]
		have the same optimal solution set.
		
		It remains to relate the minimization of $g_\lambda$ over $\mathcal U_k^n$ to the original objective $f$. For every
		${\bf x}\in\mathcal U_k^n$, we have
		\[
		\|{\bf x}\|_2^2=K.
		\]
		Therefore,
		\[
		g_\lambda({\bf x})
		=
		f({\bf x})-\lambda K,
		\qquad
		\forall {\bf x}\in\mathcal U_k^n.
		\]
		The term $-\lambda K$ is constant over $\mathcal U_k^n$ and hence does not change the set of minimizers. Consequently,
		\[
		\operatorname*{arg\,min}_{{\bf x}\in\mathcal U_k^n}
		g_\lambda({\bf x})
		=
		\operatorname*{arg\,min}_{{\bf x}\in\mathcal U_k^n}
		f({\bf x}).
		\]
		Combining this identity with the exactness over the convex hull gives
		\[
		\operatorname*{arg\,min}_{{\bf x}\in C_k}
		g_\lambda({\bf x})
		=
		\operatorname*{arg\,min}_{{\bf x}\in\mathcal U_k^n}
		f({\bf x}).
		\]
		Thus, \eqref{k-meroid-expp} is an exact reformulation of \eqref{eq:kmedoids-binary}. Since $g_\lambda({\bf x})=f({\bf x})-\lambda K$ on $\mathcal U_k^n$, the optimal values differ by the constant $-\lambda K$.
	\end{proof}


	\section{The Proposed Two-Stage Clustering Method}\label{sec:method}
	
	In this section, we present a two-stage clustering method (TSCM) for graph clustering.	The overall procedure is summarized in Algorithm~\ref{alg:two-stage}. Stage 1 solves the medoid-selection problem, for which we use a Penalty-Continuation Projected Gradient (PCPG) solver in Algorithm~\ref{alg:pcpg-solver}. Stage 2 serves as a unified decoding procedure that converts the selected medoids into a graph partition by assigning each vertex to its nearest centroid according to a two-level distance criterion, so that every vertex is assigned to exactly one cluster.
	
	Figure~\ref{fig-process} visually presents the complete execution process of Algorithm~\ref{alg:two-stage}. Given an example input image, the algorithm first enters Stage 1---Medoid Identification, which identifies two representative medoid points, namely vertices 3 and 6. Subsequently, the algorithm proceeds to Stage 2, where it performs clustering partitioning based on the two determined medoids, and finally outputs two clustering results, denoted as $C_{1}=\{1,2,3\}$ and $C_{2}=\{4,5,6,7,8,9\}$. The entire flow clearly illustrates the progressive process from the original input to the final clustering output.
	
	\begin{algorithm}[htbp]
		\caption{Two-Stage Quadratic $K$-Medoids Graph Clustering (TSCM)}
		\label{alg:two-stage}
		\begin{algorithmic}[1]
			\State \textbf{Input:} number of vertices $n$, number of clusters \(K\), $\beta$, $\alpha$
			\State \textbf{Stage 1: Medoid Identification}
			\State \quad Solve the quadratic $K$-medoids problem via the Penalty-Continuation Projected Gradient (PCPG) solver in Algorithm~\ref{alg:pcpg-solver};
			\State \quad to obtain \(K\) medoid nodes \(\mathcal{M} = \{m_1,\dots,m_K\}\);
			\State \textbf{Stage 2: Cluster Formation}
			\State \quad \textbf{for} each node \(v \in V\) \textbf{do}
			\State \quad \quad Compute Jaccard distance \(d_J(v, m_k)\) to each medoid \(m_k \in \mathcal{M}\);
			\State \quad \quad Assign \(v\) to the cluster \(C_k\) whose medoid is the closest;
			ties are broken by the shortest-path distance
			\State \quad \textbf{end}
			\State \textbf{Output:} Clusters \(C_1,\dots,C_K\)
		\end{algorithmic}
	\end{algorithm}
	
	\begin{figure}[htp]
		\centering
		\includegraphics[width=1\textwidth]{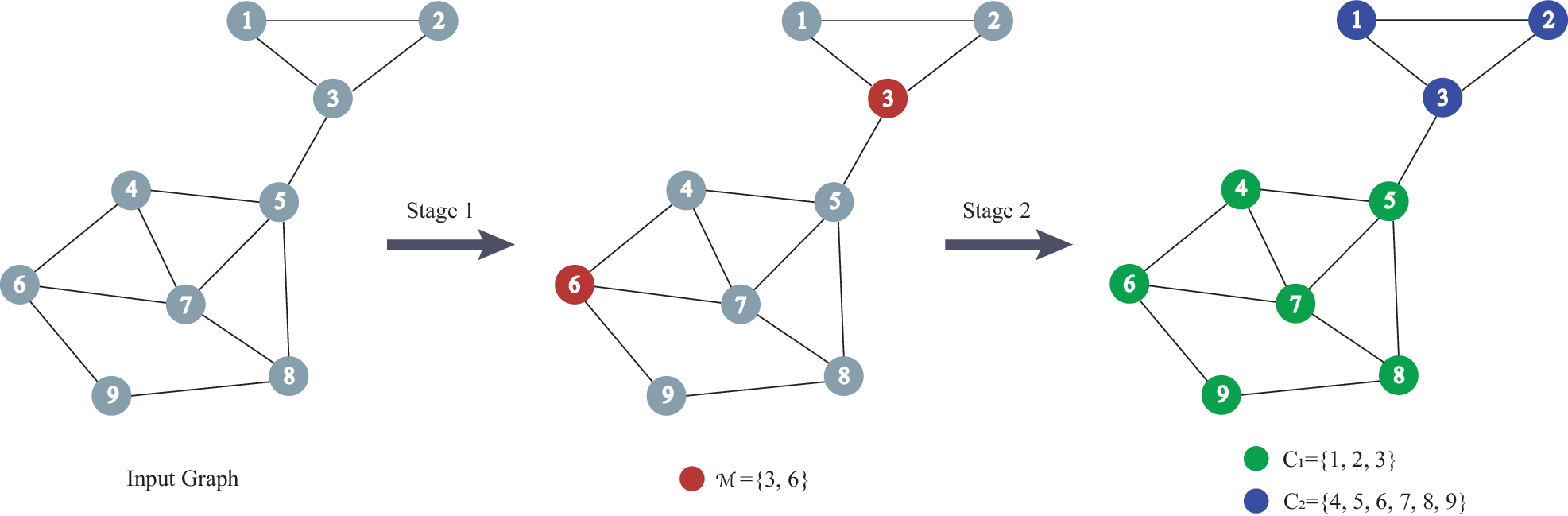}\\
		\caption{Schematic illustration of the proposed two-stage clustering method. The left part shows the original input data (unclassified); the middle part highlights the two medoid points identified in Stage 1 (in red); the right part displays the final two clusters (in blue and green) obtained in Stage 2 based on these medoids.}
		\label{fig-process}
	\end{figure}
	
	\subsection{Stage 1: Medoid Identification via PCPG}
	\label{sec:stage1}
	We solve the penalized problem over the convex-hull relaxation in~\eqref{k-meroid-expp} using a penalty-continuation projected-gradient (PCPG) scheme. Starting from a refined multi-start initialization, PCPG progressively increases the penalty parameter and approximately solves each fixed-penalty subproblem by a BB projected-gradient method with a projected Armijo line search \cite{zheng2026exact}. Binary candidate solutions are retained along the continuation path according to a prescribed incumbent criterion. The complete procedure is summarized in Algorithm~\ref{alg:pcpg-solver}.
	
	\begin{algorithm}[H] 
		\caption{Penalty-Continuation Projected Gradient (PCPG) Solver} 
		\label{alg:pcpg-solver} 
		\begin{algorithmic}[1] 
			\State \textbf{Input:} $D$, $K$, $\alpha>0$, $\beta>0$, $\lambda_0>0$, $\lambda_{\max}>0$, $\varepsilon_{\rm gap}>0$,  $N_{\rm init}$, $N_{\rm ref}$, criterion $\chi\in\{\mathrm{obj},\mathrm{assign}\}$. 
			\State \textbf{Initialization:} 
			\State \quad Generate $N_{\rm init}$ medoid sets by farthest-first initialization and refine each by $N_{\rm ref}$ $K$-medoids passes.
			\State \quad Select the candidate with the smallest nearest-medoid assignment cost as $\mathcal M^0$, and construct the starting point ${\bf x}^0\in\{0,1\}^n$ from $\mathcal M^0$.
			\State \quad Initialize ${\bf x}^{\rm bi}$ as the best binary candidate among the $N_{\rm init}$ initialization candidates according to the score $S_\chi$ defined in \eqref{eq:incumbent-score}.
			\State \quad Set $\ell:=0$ and $\lambda_{l}:=\lambda_0$.
			 
			\While{true} 			 
			\State Approximately solve the fixed-$\lambda$ penalized problem \eqref{k-meroid-expp-l} by the BB projected-gradient method with a projected Armijo line search, initialized at ${\bf x}^{\ell}$, to obtain ${\bf x}^{\ell+1}$.
			\State Compute the binary gap according to \eqref{eq:binary-gap}.
			\State Generate the binary candidate
			$\widehat{\bf x}^{\,\ell+1}$ and construct
			the corresponding medoid set
			$\widehat{\mathcal M}^{\,\ell+1}$.
			\State Update ${\bf x}^{\rm bi}$ if
			$\widehat{\mathcal M}^{\,\ell+1}$
			yields a smaller value of the selected criterion $S_\chi$.
			\If{$\operatorname{gap}({\bf x}^{\ell+1})
				\leq\varepsilon_{\rm gap}$
				\textbf{or}
				$\lambda\geq\lambda_{\max}$}
			\State \textbf{break}
			\EndIf
			\State Compute $r_{\ell+1}$ and choose
			$\tau_{\ell+1}$ according to
			\eqref{eq:normalized-gap} and \eqref{eq:adaptive-tau}.
			
			\State Update $\lambda$ according to
			\eqref{eq:lambda-continuation},
			and set $\ell:=\ell+1$.
			\EndWhile
			
			\State \textbf{Output:}
			incumbent binary medoid-selection vector ${\bf x}^{\rm bi}$.
			\end{algorithmic} 
		\end{algorithm}
				
	Specifically, for the initialization strategy, since the model is nonconvex, different initial points may lead to substantially different local minima, and hence the quality of initialization can directly affect the final solution. To obtain a reliable starting point, we employ a multi-start initialization scheme combining the farthest-first strategy \cite{le2020farthest,gonzalez1985clustering} with a few $K$-medoids assignment--update passes \cite{ParkJun2009}. For a medoid set $\mathcal M$, we define the nearest-medoid assignment cost as
    \begin{equation}\label{eq:incumbent-assignment}
	    A(\mathcal M)=\sum_{i=1}^{n}\min_{m\in\mathcal M}d_{im}.
    \end{equation}
	Among the resulting candidate medoid sets, the one with the smallest $A(\mathcal M)$ is selected as $\mathcal{M}^{0}$, from which the initial point ${\bf x}^0$ is constructed.
	
	In addition to the initial point, we maintain an incumbent binary solution ${\bf x}^{\rm bi}$ to retain the best binary candidate encountered during the algorithm. For this purpose, we consider two criteria. The first is the original quadratic medoid objective
	\begin{equation}\label{eq:incumbent-original}
		F(\mathcal M)=\beta\sum_{i\in\mathcal M}(D{\bf 1})_i-\frac{\alpha}{2}\sum_{i,j\in\mathcal M}d_{ij},
	\end{equation}
	and the second is the nearest-medoid assignment cost $A(\mathcal M)$ defined in \eqref{eq:incumbent-assignment}. Accordingly, for a prescribed criterion $\chi\in\{\mathrm{obj},\mathrm{assign}\}$, we define the incumbent score as 	\begin{equation}\label{eq:incumbent-score}
		S_\chi(\mathcal M)
		=
		\begin{cases}
			F(\mathcal M), & \chi=\mathrm{obj},\\
			A(\mathcal M), & \chi=\mathrm{assign},
		\end{cases}
	\end{equation}
	Among the refined initialization candidates, the one with the smallest value of $S_\chi$ is selected as the initial incumbent ${\bf x}^{\rm bi}$.
	
	For the penalty parameter $\lambda$, we adopt an adaptive continuation scheme. Starting from an initial value $\lambda_0>0$, a sequence of fixed-penalty subproblems is approximately solved while progressively increasing $\lambda$. For each fixed $\lambda_\ell$, setting $\lambda=\lambda_\ell$ in problem~\eqref{k-meroid-expp} yields the subproblem
	\begin{equation}\label{k-meroid-expp-l}
		\min_{{\bf x}\in C_k}
		g_{\lambda_\ell}({\bf x}).
	\end{equation}
	This subproblem is approximately solved by the BB projected-gradient method with the projected Armijo line search. Specifically, since $D$ is symmetric, the gradient of $g_{\lambda_\ell}$ is $\nabla g_\lambda({\bf x})=\beta D{\bf 1}-\alpha D{\bf x}-2\lambda{\bf x}$.
	At each inner iteration, the projected-gradient update takes the form
	\begin{equation}\label{eq:pg-update}
		{\bf x}^{t+1}=\Pi_{C_K}\left({\bf x}^{t}-\eta_t\nabla g_{\lambda_\ell}({\bf x}^{t})
		\right),
	\end{equation}
	where $\eta_t$ is initialized using a safeguarded BB stepsize and adjusted by a projected Armijo backtracking line search. The projection onto $C_K$ is given by $\Pi_{C_K}({\bf v})=\left[{\bf v}-\theta{\bf 1}\right]_{0}^{1}$, where $\theta\in\mathbb R$ is chosen such that ${\bf 1}^{\top}\left[{\bf v}-\theta{\bf 1}\right]_{0}^{1}=K$.
	Since the left-hand side is monotone in $\theta$, the scalar $\theta$ can be efficiently determined by bisection.
	
	To measure the degree of binarity of the current iterate, we define the binary gap as
	\begin{equation}\label{eq:binary-gap}
		\operatorname{gap}({\bf x})
		:=
		K-\|{\bf x}\|_2^2.
	\end{equation}
	For any ${\bf x}\in C_k$, $\operatorname{gap}({\bf x})\geq0$, with equality if and only if ${\bf x}\in\mathcal U_k^n$. Thus, a smaller gap indicates that the current iterate is closer to satisfying the binary constraint. Following the homotopy strategy for penalty-parameter updates in \cite{shao2020binary,liu2024expp1}, we progressively increase the penalty parameter. To adapt the increase to the degree of binarity of the current iterate, we normalize the gap as
	\begin{equation}\label{eq:normalized-gap}
		r_{\ell+1}:=\frac{\operatorname{gap}({\bf x}^{\ell+1})}{K},
	\end{equation}
	and use it to determine the update factor
	\begin{equation}\label{eq:adaptive-tau}
		\tau_{\ell+1}=
		\begin{cases}
			2, & r_{\ell+1}>0.8,\\
			1.5, & 0.2<r_{\ell+1}\leq0.8,\\
			1.25, & r_{\ell+1}\leq0.2.
		\end{cases}
	\end{equation}
	The penalty parameter is then updated by
	\begin{equation}\label{eq:lambda-continuation}
		\lambda_{\ell+1}
		=
		\min\left\{
		\lambda_{\max},
		\tau_{\ell+1}\lambda_\ell
		\right\}.
	\end{equation}
	Hence, $\lambda$ is increased more aggressively when the current iterate is highly fractional and more moderately as it approaches a binary point. The outer iteration terminates when $\operatorname{gap}({\bf x}^{\ell+1})\leq\varepsilon_{\rm gap}$
	or the current penalty parameter $\lambda_\ell\geq\lambda_{\max}$.
	
	After solving each fixed-penalty subproblem, a feasible binary candidate $\widehat{\bf x}^{\,\ell+1}\in\mathcal U_k^n$ is obtained by setting the $K$ largest components of ${\bf x}^{\ell+1}$ to one and the remaining components to zero. Let $\widehat{\mathcal M}^{\,\ell+1}$ denote the corresponding medoid set. The candidate is then compared with the current incumbent according to $S_\chi$ in \eqref{eq:incumbent-score}, and ${\bf x}^{\rm bi}$ is updated whenever a smaller incumbent score is obtained.

	\subsection{Stage 2: Cluster Formation and Assignment}
	\label{sec:stage2}
	
	Having identified the set of medoids \(\mathcal{M}=\{m_1,\dots,m_K\}\) in Stage~1,
	we now assign each node \(v\in V\) to exactly one cluster.  
	The assignment is based on a two-level distance criterion:
	\begin{enumerate}
		\item Primary: Jaccard distance \(d_{v,m_{k}}\);
		\item Secondary: shortest-path distance \(d_{\mathrm{sp}}(v,m_k)\).
	\end{enumerate}
	Note that, for clarity, both $d_{i,j}$ used here and $d_{ij}$ in Section~\ref{subsubsec:jaccard-distance} denote the Jaccard distance between vertices $i$ and $j$. The Jaccard distance has been introduced in Section~\ref{subsubsec:jaccard-distance}; the shortest-path distance \(d_{\mathrm{sp}}(u,v)\) is defined as the minimum number of edges among all paths connecting \(u\) and \(v\) in the graph. The motivation for this two-level design is that Jaccard distance captures neighborhood overlap, which is a natural similarity measure for graph nodes, while shortest-path distance provides a fallback when Jaccard distances are non-discriminative.
	
	Concretely, for a given vertex \(v\), we define the set of closest medoids under Jaccard distance:
	\[
	\mathcal{M}^*(v) = \arg\min_{m_k\in\mathcal{M}} d_{v,m_{k}}.
	\]
	If \(|\mathcal{M}^*(v)| = 1\), we assign \(v\) to that unique cluster. If \(|\mathcal{M}^*(v)| > 1\), we assign \(v\) to the cluster whose medoid minimizes the shortest-path distance:
	\[
	c(v) = \arg\min_{m_k\in\mathcal{M}^*(v)} d_{\mathrm{sp}}(v,m_k).
	\]
	This rule is applied independently for every node; the resulting partition defines the final clusters \(C_1,\dots,C_K\).
	
	\subsubsection{Illustrative Example: Why a Secondary Criterion is Necessary}
	\label{sec:tie-example}
	
		\begin{figure}[htp]
		\centering
		\includegraphics[width=0.9\textwidth]{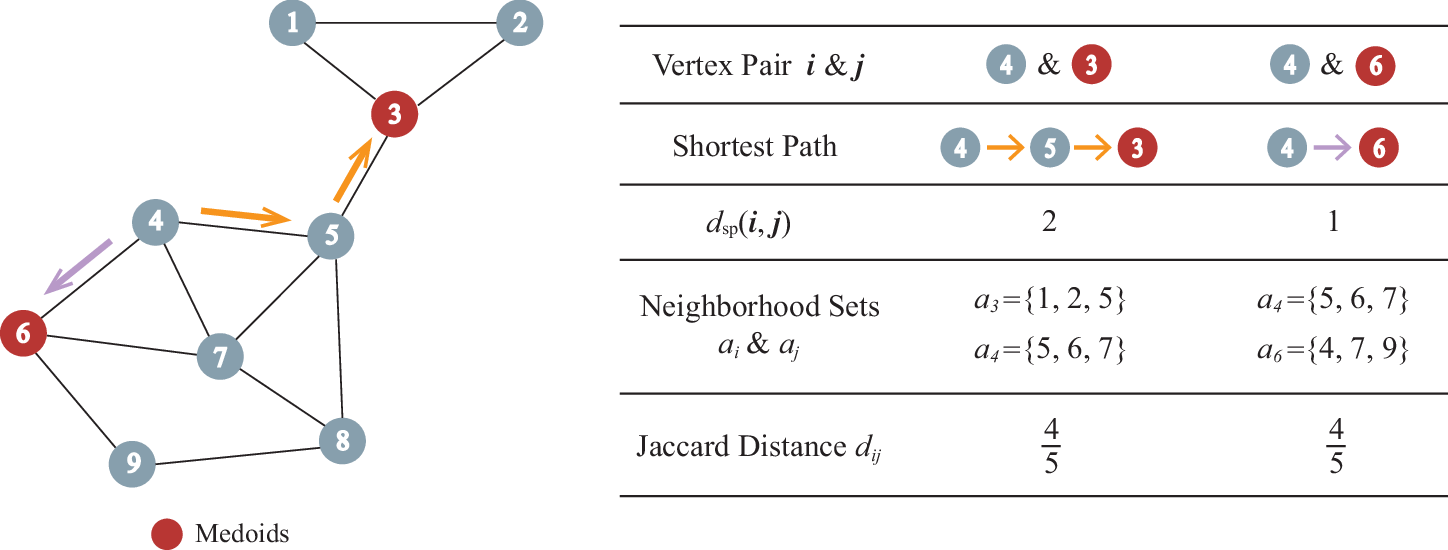}\\
		\caption{Example graph and related metric.}
		\label{fig:tie-example}
	\end{figure}
	
	Consider the graph fragment shown in Figure~\ref{fig:tie-example}. Medoids are vertices \(3\) and \(6\), and we need to assign vertex \(4\). As indicated by the table on the right-hand side of the figure, the neighborhood sets of vertices $3$, $4$, and $6$ are $\{1,2,5\}$, $\{5,6,7\}$, and $\{4,7,9\}$ respectively. The Jaccard distance between vertices $4$ and $3$ is computed as 
	$$d_{34}=1-\frac{|\{1,2,5\}\cap\{5,6,7\}|}{|\{1,2,5\}\cup\{5,6,7\}|}=\frac{4}{5},$$
	and the distance between vertices $4$ and $6$ is $$d_{46}=1-\frac{|\{5,6,7\}\cap\{4,7,9\}|}{|\{5,6,7\}\cup\{4,7,8\}|}=\frac{4}{5}.$$
	Then the Jaccard distances satisfy \(d_{34}=d_{46}\). However, the shortest-path distances are \(d_{\mathrm{sp}}(3,4)=2\) and \(d_{\mathrm{sp}}(4,6)=1\). Using only Jaccard distance would lead to an arbitrary choice, whereas incorporating shortest-path distance correctly places \(4\) in \(6\)'s cluster, which is topologically closer. This example motivates the two-level rule in our algorithm.

	
	\section{Numerical Experiments}
	
	In this section, we evaluate the performance of our algorithm on different data sets. We compare the solution quality of our algorithm TSCM, the Boltzmann machine and Gurobi run on a 4-core/8-thread machine. All experiments are performed in Windows 10 on an Intel(R) Core(TM) i7-1065G7 CPU at 1.30 GHZ with 16.0 GB of RAM. The code for all methods is written in MATLAB R2022a, and all graphs are generated using Python 3.82
	. To conclude, we find that our method demonstrates the ability to recover clusters not only in synthetic graphs but also in real-world graphs.
	
	\subsection{Evaluation Measures}
	\label{subsec:evaluation-setup}
	For the experiments in this chapter, we use ARI and density metrics as complementary measures: density assesses the structural quality of the resulting partition, and ARI assesses the label recovery accuracy relative to the true partition.
	\paragraph{Density-based Structural Measures}
	To evaluate the quality of a graph clustering result, we use three density-based quantities: the overall density, the intra-cluster density, and the inter-cluster density. 
	As shown in \cite{cq-1,cq-2,Teng_2025}, these density-based
	criteria provide an effective way to assess clustering quality and can be more informative than the widely used modularity measure \cite{cq-3}.
	

	Suppose that the
	vertices in graph $G$ are partitioned into $K$ clusters. Denote by $n^{(i)}$ the number
	of vertices in cluster $i$. Let $|E_{ii}|$ be the number of edges whose two
	endpoints both belong to cluster $i$, and let $|E_{ij}|$ be the number of
	edges connecting a vertex in cluster $i$ and a vertex in cluster $j$,
	where $i\neq j$.
	
	The overall density of the graph is defined as
	\begin{equation}\label{eq:overall-density}
		\kappa
		=
		\frac{|E(G)|}{0.5\times n(n-1)}.
	\end{equation}
	For cluster $i$, the intra-cluster density is defined by $\kappa_{\mathrm{intra}}^{(i)}=|E_{ii}|/\left(0.5\times n^{(i)}(n^{(i)}-1)\right)$.
	For two distinct clusters $i$ and $j$, the inter-cluster density is defined by $\kappa_{\mathrm{inter}}^{(ij)} = |E_{ij}|/\left(n^{(i)}n^{(j)}\right)$.
	
	The mean intra-cluster density and the mean inter-cluster density are then given by
	\begin{equation}\label{eq:mean-density}
		\bar{\kappa}_{\mathrm{intra}}
		=
		\frac{1}{K}
		\sum_{i=1}^{K}
		\kappa_{\mathrm{intra}}^{(i)},
		\qquad
		\bar{\kappa}_{\mathrm{inter}}
		=
		\frac{1}{\frac{1}{2}K(K-1)}
		\sum_{i=1}^{K}
		\sum_{j=i+1}^{K}
		\kappa_{\mathrm{inter}}^{(ij)} .
	\end{equation}
	
	
	These quantities characterize how densely connected the vertices are within clusters and how sparsely connected different clusters are. 
	A high quality clustering groups vertices into clusters such that, on average, the links between vertices within these clusters are denser than the links between vertices in different clusters \cite{cq-1,cq-2}.
	Therefore, a high-quality clustering is
	expected to satisfy
	\begin{equation}\label{cq-standard}
		\bar{\kappa}_{\mathrm{inter}}
		<
		\kappa
		<
		\bar{\kappa}_{\mathrm{intra}}.
	\end{equation}
	This criterion will be used in our numerical experiments to assess the
	quality of the obtained clustering results.
	
	\begin{figure}[htp]
		\centering
		\includegraphics[width=1\textwidth]{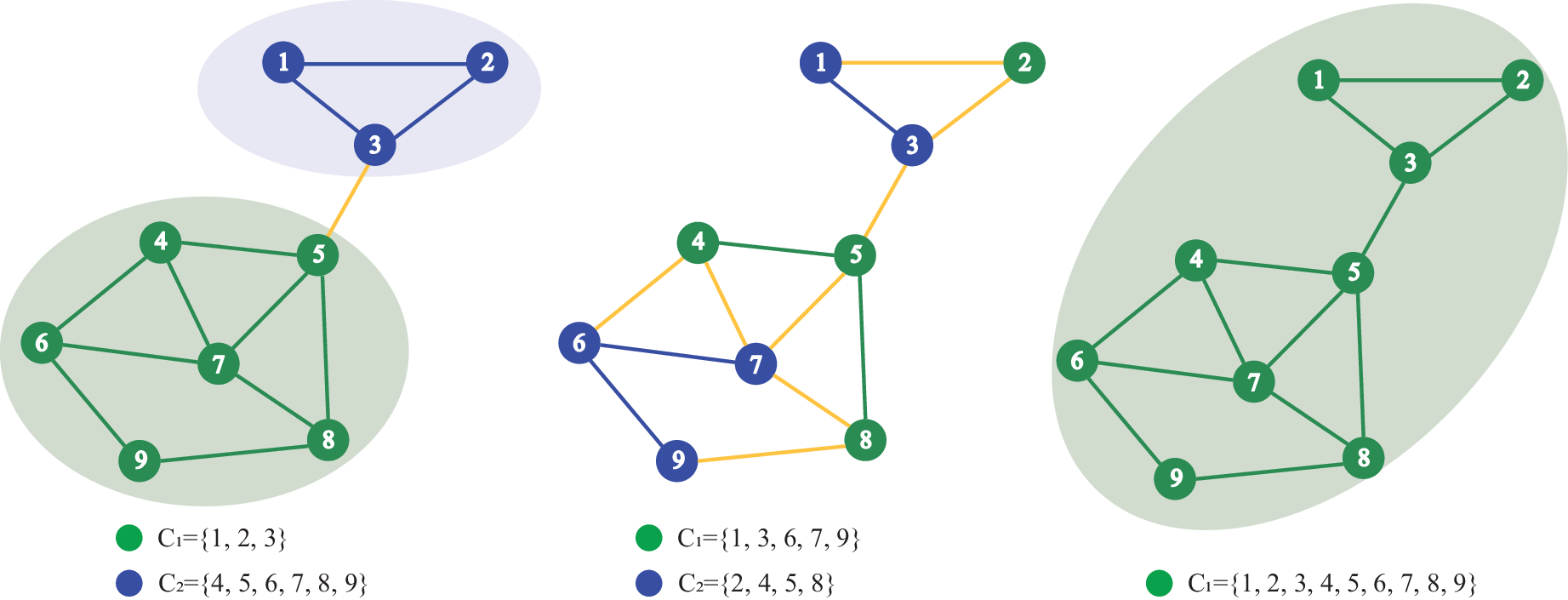}\\
		\caption{Example graph and related metric.}
		\label{fig:evaluation-example}
	\end{figure}
	
	To better understand inequality \eqref{cq-standard} satisfied by the above evaluation metrics, consider the network in  Figure~\ref{fig:evaluation-example}, which has three subfigures (left, middle, right) showing different partitions of the same network. Green edges are inside community $C_{1}$, blue edges inside $C_{2}$, and yellow edges are between communities. The overall density is
	\begin{equation*}
		\kappa=\frac{13}{0.5\times9\times8}\approx0.36.
	\end{equation*}
	
	In the left subfigure, the network is partitioned into $C_{1}=\{1,2,3,4,5\}$ and $C_{2}=\{6,7,8,9\}$. The intra-cluster densities are 
	\begin{equation*}
		\kappa^{C_{1}}_{intra}=\frac{3}{0.5\times3\times2}=1,\ \kappa^{C_{2}}_{intra}=\frac{9}{0.5\times9\times8}=\frac{1}{4}.
	\end{equation*}
	so $\bar{\kappa}_{intra}=(1+1/4)/2\approx0.63$.
	The inter-cluster density is $\bar{\kappa}_{inter}=1/(3\times6)\approx0.06$. This partition satisfies $\bar{\kappa}_{\mathrm{inter}}<\kappa<\bar{\kappa}_{\mathrm{intra}}$, thus meeting inequality \eqref{cq-standard}.
	
	In the middle subfigure, the partition is $C_{1}=\{1, 3, 6, 7, 9\}$ and $C_{2}=\{2, 4, 5, 8\}$. Then
	\begin{equation*}
		\kappa^{C_{1}}_{intra}=\frac{3}{0.5\times5\times4}=0.3,\ \kappa^{C_{2}}_{intra}=\frac{2}{0.5\times4\times3}=\frac{1}{3},
	\end{equation*}
	so $\bar{\kappa}_{intra}\approx0.32$,
	and $\bar{\kappa}_{inter}=8/(5\times4)=0.4$. Hence $\bar{\kappa}_{intra}<\kappa<\bar{\kappa}_{inter}$, violating \eqref{cq-standard}.
	
	In the right subfigure, all vertices are grouped into one community $C_{1}=\{1,2,3,4,5,6,7,8,9\}$. Then $\bar{\kappa}_{intra}=\kappa\approx0.36$ and $\bar{\kappa}_{inter}=0$, so $\bar{\kappa}_{inter}<\kappa=\bar{\kappa}_{intra}$, also failing the criterion.
	
	\paragraph{An illustrative example and limitations of density measures}
	
	Next, we use a small example to illustrate a behavior of density-based metrics, revealing the inadequacy of relying solely on such metrics. For the ground-truth partition (the leftmost partition in Figure~\ref{fig:evaluation-example}), the mean intra-cluster density is $0.625$, while the mean inter-cluster density is about $0.0556$. Starting from a partition with four misclassified vertices, we gradually correct the assignments, as shown in the Figure~\ref{fig:evaluation-example2}. 
	\begin{figure}[htp]
		\centering
		\includegraphics[width=1\textwidth]{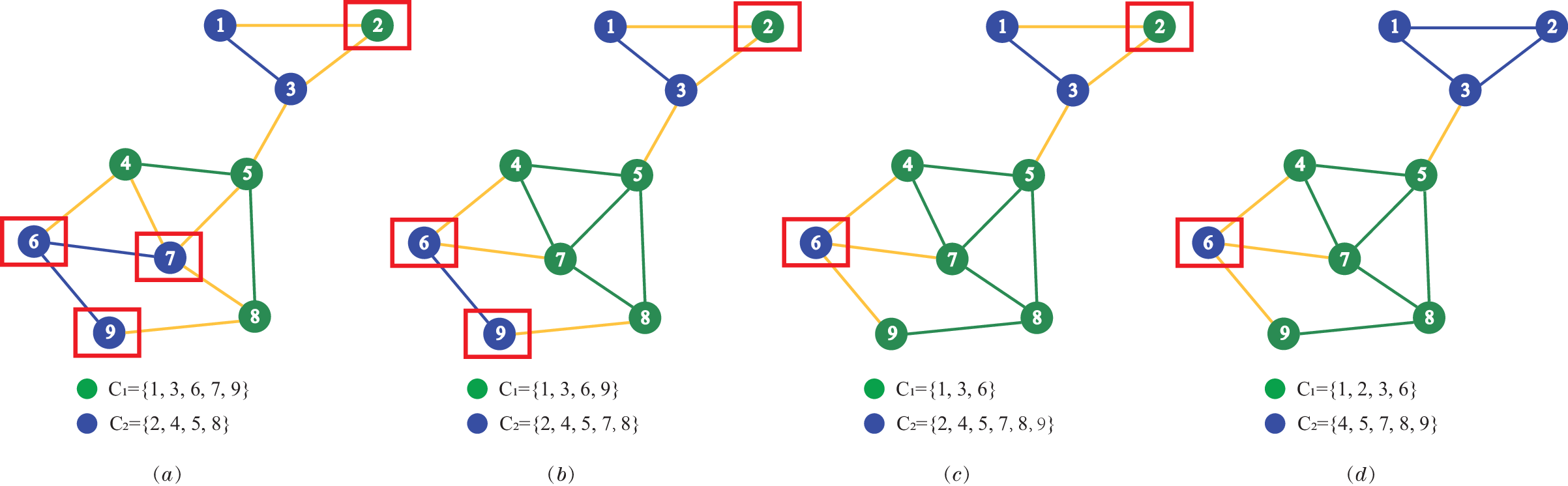}\\
		\caption{Four partitions with decreasing numbers of misclassified vertices: (a) 4 misclassifications, (b) 3, (c) 2, and (d) 1..}
		\label{fig:evaluation-example2}
	\end{figure}
	The resulting density values are listed in Table~\ref{tab:density-misclassification}. $N_{\text{mis}}$ is the number of misclassified vertices.
\begin{table}[htbp] 
	\centering 
	\caption{Mean intra-cluster and inter-cluster densities under different numbers of misclassified vertices.} \label{tab:density-misclassification} 
	\begin{tabular*}{0.9\textwidth}{@{\extracolsep{\fill}}cccccc}
			\toprule
            Metric & Figure~\ref{fig:evaluation-example2}(a) & Figure~\ref{fig:evaluation-example2}(b)  & Figure~\ref{fig:evaluation-example2}(c) & Figure~\ref{fig:evaluation-example2}(d) & Left subfigure of Figure~\ref{fig:evaluation-example} \\
            \midrule
            $N_{\text{mis}}$ & 4 & 3 & 2 & 1 & 0 \\
            \midrule
            $\bar{\kappa}_{\mathrm{intra}}$ & 0.3167 & 0.4167 & 0.3667 & 0.5500 & 0.6250 \\
            \midrule
            $\bar{\kappa}_{\mathrm{inter}}$ & 0.4000 & 0.3000 & 0.3333 & 0.2000 & 0.0556 \\
            \bottomrule
		\end{tabular*} 
	\end{table}
	
	The example shows that neither the mean intra-cluster density nor the mean inter-cluster density is necessarily monotone with respect to the number of misclassified vertices. In particular, the partition with two misclassified vertices has a lower mean intra-cluster density and a higher mean inter-cluster density than the partition with three misclassified vertices. This occurs because different vertices contribute unequally to the internal connectivity of a cluster. Misclassifying a structurally important vertex with a large within-community degree may remove more intra-cluster edges than misclassifying several peripheral vertices. Moreover, the normalization terms in the density measures change with the cluster sizes. Consequently, a partition with fewer misclassified vertices may occasionally have less favorable density values. 
	
	Therefore, closeness to the density values associated with the ground-truth partition alone is insufficient for determining the quality of an estimated partition. The density measures describe the structural connectivity of the detected clusters, but they do not directly measure whether individual vertices are assigned to their correct communities.
	
	\paragraph{Adjusted Rand index}
	
	The density-based measures evaluate the cohesion and separation of the detected clusters from the graph-structure perspective. For synthetic graphs, where the ground-truth community labels are available, we further employ the adjusted Rand index (ARI) \cite{hubert1985comparing} to evaluate the agreement between the detected and ground-truth partitions.

    The ARI is a chance-corrected version of the Rand index \cite{rand1971objective} and is invariant under permutations of cluster labels; moreover, the ground-truth and detected partitions need not contain the same number of clusters. An ARI equal to $1$ indicates exact agreement between the two partitions, a value close to $0$ indicates an agreement comparable to that expected by chance, and a negative value indicates an agreement lower than the random expectation. Further properties of the ARI can be found in \cite{steinley2004properties}.
	
	Consequently, the density measures and the ARI provide complementary information. The former describe the internal connectivity and separation of the detected clusters, whereas the latter directly evaluates the recovery of the ground-truth vertex assignments.

	\subsection{Synthetic Graphs}
	
	In this section, we conduct experiments on 12 different synthetic graphs, generated using two distinct models: the Planted Partition Model (PPM) \cite{PPM} and the Stochastic Block Model (SBM) \cite{SBM}. These graphs have known clustering properties, with details provided in Table~\ref{table-1}.
	
	\begin{table}[H]
		\centering
		\caption{Details of nine synthetic graphs.}
		\label{table-1}
		\begin{tabular*}{0.9\textwidth}{@{\extracolsep{\fill}}ccccccc}
			\toprule
			Model type & Graph name & ${\rm P_{intra}}$ & ${\rm P_{inter}}$ & $K$  & $n$    & Cluster sizes \\ \midrule
			PPM        & G1\_PPM    & 0.90          & 0.100          & 5  & 250  & 50            \\ \midrule
			PPM        & G2\_PPM    & 0.85         & 0.150         & 5  & 250  & 50            \\ \midrule
			PPM        & G3\_PPM    & 0.80          & 0.200          & 5  & 250  & 50            \\ \midrule
			SBM        & G1\_SBM    & 0.90          & 0.050         & 30 & 3000 & {[}25,200{]}  \\ \midrule
			SBM        & G2\_SBM    & 0.90          & 0.075         & 30 & 3000 & {[}25,200{]}  \\ \midrule
			SBM        & G3\_SBM    & 0.90          & 0.100         & 30 & 3000 & {[}25,200{]}  \\ \midrule
			SBM        & G4\_SBM   & 0.85         & 0.050         & 30 & 3000 & {[}25,200{]}  \\ \midrule
			SBM        & G5\_SBM   & 0.85         & 0.075         & 30 & 3000 & {[}25,200{]}  \\ \midrule
			SBM        & G6\_SBM    & 0.85         & 0.100          & 30 & 3000 & {[}25,200{]}  \\ \midrule
			SBM        & G7\_SBM    & 0.80          & 0.050         & 30 & 3000 & {[}25,200{]}  \\ \midrule
			SBM        & G8\_SBM    & 0.80          & 0.075         & 30 & 3000 & {[}25,200{]}  \\ \midrule
			SBM        & G9\_SBM    & 0.80          & 0.100          & 30 & 3000 & {[}25,200{]}  \\ \bottomrule
		\end{tabular*} 
	\end{table}%
	
	First, we test on three PPM graph models, each containing $250$ vertices and $5$ clusters of $50$ vertices each. Second, we test on nine SBM graph models—each comprising $3000$ vertices and $30$ clusters, with cluster sizes ranging from $25$ to $200$. These graphs are generated with the following intra-/inter-cluster edge probability parameters (${\rm P_{intra}}$/${\rm P_{inter}}$) as shown in Table~\ref{table-1}. For each intra-cluster edge probability, we generate graphs using two different inter-cluster edge probabilities. All synthetic graphs are generated using the NetworkX library \cite{gc-no26} in Python. As the intra-cluster edge probability decreases, the inter-cluster edge probability increases, and the cluster sizes vary, the clustering task becomes more challenging; concurrently, an increase in the number of vertices further escalates the computational difficulty.
	
	Given that our proposed framework consists of two stages, while the compared methods (Gurobi and the Boltzmann machine) are only used to solve the first stage, we divide our experimental comparisons into two clear levels for a comprehensive and fair evaluation:
	
	(1) \textbf{Optimization Performance of Medoid Selection} At this level, we directly compare the mathematical optimization capabilities of each method in solving Model (1). Specifically, we record and compare the objective function values $f(\bf{x})$ obtained by Gurobi, the Boltzmann machine, and our proposed PCPG method under the same time limit, along with the solution time (in seconds). This metric objectively measures the accuracy and efficiency of each optimizer in selecting a high‑quality Medoid set within the combinatorial space, independent of the subsequent assignment step.
	
	(2) \textbf{Clustering Quality after Unified Assignment} After obtaining the binary vectors $\bf{x}$ from all methods, we uniformly apply the same second‑stage nearest‑Medoid assignment strategy proposed in this paper to produce the final cluster partitions. Since the second stage is identical for all compared methods, the subsequent clustering quality metrics purely reflect the quality of the Medoid sets selected in the first stage. At this level, we adopt two complementary metrics to assess clustering quality:
	\begin{itemize}
	\item Since PPM and SBM graphs are generated with explicit intra-cluster edge probability ${\rm P_{intra}}$ and inter-cluster edge probability ${\rm P_{inter}}$, we first employ the mean intra-cluster density $\bar{\kappa}_{\mathrm{intra}}$ and mean inter-cluster density $\bar{\kappa}_{\mathrm{inter}}$ to assess how well the algorithms recover the original generative mechanism. Ideally, $\bar{\kappa}_{\mathrm{intra}}$ should approximate ${\rm P_{intra}}$, and $\bar{\kappa}_{\mathrm{inter}}$ should approximate ${\rm P_{inter}}$ \cite{miasnikof2024graph}. In the results presentation, we use $\epsilon_{\text{intra}}$ (or $\epsilon_{\text{inter}}$) to evaluate clustering quality. Here, $\epsilon_{\text{intra}}$ (or $\epsilon_{\text{inter}}$) denotes the absolute difference between the algorithmically computed $\bar{\kappa}_{\mathrm{intra}}$ (or $\bar{\kappa}_{\mathrm{inter}}$) and the true graph values ${\rm P_{intra}}$ (or ${\rm P_{inter}}$). We consider smaller values of this metric to indicate better clustering structure quality.
	
	\item As described in Section~\ref{subsec:evaluation-setup}, to complement the density-based structural measures, we also use the Adjusted Rand Index (ARI) to evaluate the agreement between the estimated partition and the ground-truth partition. Higher ARI values indicate better recovery of the true communities, with an ARI of 1 corresponding to perfect recovery.
	
	\end{itemize}
	
	For all PPM and SBM instances, we employ the same PCPG framework, including the multi-start farthest-first initialization, short K-medoids initialization refinement, adaptive penalty continuation, and BB projected-gradient inner solver. The two graph models differ only in the criterion used to retain the incumbent binary solution. For the PPM instances, the incumbent is selected according to the original quadratic medoid objective, whereas for the SBM instances, the nearest-medoid assignment cost is used. The latter criterion is consistent with the second stage of the clustering procedure, in which each vertex is assigned to its nearest selected medoid.
	
\subsubsection{Results on PPM Graph Models}
In this section, we evaluate all competing methods on the three PPM-generated graphs. 
These graphs are relatively small and balanced, allowing us to initially verify the fundamental 
effectiveness of each optimizer without the interference of scale imbalance. For a given test instance, all three methods use the same values of $\alpha$ and $\beta$. Specifically, for PPM1 and PPM2, we set $\beta=K/n$, $\alpha=4\beta$, whereas for PPM3, we use  \begin{equation} \label{eq: ppm3-alpha} 
	\beta=\frac{K}{n}, \qquad \alpha= \frac{2n}{K-1}\beta. 
\end{equation} 

\paragraph{Optimization Performance of Medoid Selection}

To ensure a fair comparison, TSCM, Gurobi, and BM are treated as alternative solvers for the medoid selection problem in the first stage. The raw objective values $f(\bf{x})$ reported in Table~\ref{tab:compare-PPM-1} are computed by substituting the final binary medoid vector into \eqref{eq:kmedoids-binary}. 

	\begin{table}[ht]
	\centering
	\caption{Comparison of optimization performance on PPM models with $n=250$ and $K=5$.}
	\label{tab:compare-PPM-1}
	\setlength{\tabcolsep}{4.5pt} 
		\begin{tabular*}{0.9\textwidth}{@{\extracolsep{\fill}}ccccccc}
			\toprule
			\multirow{2.4}{*}{method} & \multicolumn{2}{c}{G1\_PPM} & \multicolumn{2}{c}{G2\_PPM} & \multicolumn{2}{c}{G3\_PPM} \\
			\cmidrule(r){2-3} \cmidrule(r){4-5} \cmidrule(r){6-7} 
			& $f(\bf{x})$& CPU time(s)     & $f(\bf{x})$ & CPU time(s) &  $f(\bf{x})$ & CPU time(s) \\
			\midrule
			TSCM & 20.1034 &  0.17  & 19.4492 & 0.18 & -1.7343 & 0.19\\
			\midrule
			Gurobi & 19.9031& 0.45 & 19.4492 & 1.76 & -1.7925 & 0.25 \\
			\midrule
			BM & 19.9090& 0.10 & 19.4570 & 0.11  & -1.8038 & 0.08\\
			\bottomrule
		\end{tabular*}%
	%
\end{table}%

The optimization results on the three PPM instances are reported in Table~\ref{tab:compare-PPM-1}. Since  problem~\eqref{eq:kmedoids-binary} is a minimization problem, a smaller value of $f(\bf{x})$ indicates a better feasible solution. The results show that the PCPG-based TSCM provides a favorable balance between solution quality and computational efficiency. 
On G1\_PPM, TSCM obtains an objective value of $20.1034$, compared with $19.9031$ for Gurobi and $19.9090$ for BM. Although the objective value obtained by TSCM is slightly larger than the best-known value, its relative deviation is only about $1.01\%$. Moreover, TSCM requires only $0.17$s, substantially less than the $0.45$s required by Gurobi.
On G2\_PPM, TSCM and Gurobi attain the same best-known objective value of $19.4492$, while BM gives a slightly larger value of $19.4570$. More importantly, TSCM requires only $0.18$s, whereas Gurobi takes $1.76$s, demonstrating a substantial computational advantage while preserving the same objective quality.
On the more challenging G3\_PPM instance, BM obtains the smallest objective value of $-1.8038$, followed by Gurobi with $-1.7925$ and TSCM with $-1.7343$. Thus, TSCM does not achieve the best Stage~1 objective value on this instance. Nevertheless, as shown by the clustering results in Table~\ref{tab:compare-PPM-2}, this difference in objective value does not lead to a deterioration in the final graph partition. These results indicate that the PCPG-based solver emphasizes computational efficiency and the identification of sufficiently high-quality medoids rather than the pursuit of marginal improvements in the Stage~1 objective value.

\paragraph{Clustering Quality after Unified Assignment}
	
After the first stage, the medoid sets obtained by the three methods are converted into graph partitions via the identical nearest‑medoid assignment rule in the second stage. In this subsection, we evaluate the clustering quality of the final graph partitions using the ARI, $\epsilon_{intra}$ and $\epsilon_{inter}$, and also report the total running time ${\rm T_{total}}$.

	\begin{table}[ht]
		\centering
		\caption{Comparison of some clustering metrics on PPM models with $n = 250$ and $K = 5$.}
		\label{tab:compare-PPM-2}%
		\setlength{\tabcolsep}{3.7pt} 
			\begin{tabular*}{0.91\textwidth}{@{\extracolsep{\fill}}ccccccccccccc}
				\toprule
				\multirow{2.4}{*}{method} & \multicolumn{4}{c}{G1\_PPM} & \multicolumn{4}{c}{G2\_PPM} & \multicolumn{4}{c}{G3\_PPM} \\
				\cmidrule(r){2-5} \cmidrule(r){6-9} \cmidrule(r){10-13} 
				& $\epsilon_{intra}$ & $\epsilon_{inter}$ & ARI & ${\rm T_{total}}$(s)     & $\epsilon_{intra}$ & $\epsilon_{inter}$ & ARI  & ${\rm T_{total}}$(s) & $\epsilon_{intra}$ & $\epsilon_{inter}$ & ARI  & ${\rm T_{total}}$(s) \\
				\midrule
				TSCM & 0 & 0& 1  & 0.20  & 0 & 0& 1 & 0.21 & 0 & 0& 1 & 0.23  \\
				\midrule
				Gurobi & 0 & 0& 1 &  0.46 & 0 & 0& 1 & 1.79 & 0 & 0& 1 & 0.26 \\
				\midrule
				BM & 0  & 0& 1 & 0.14  &0 & 0& 1 & 0.14  & 0.0074 & 0.0018& 0.9800 & 0.09 \\
				\bottomrule
			\end{tabular*}%
		
	\end{table}%
	
	The clustering results on the three PPM instances are reported in Table~\ref{tab:compare-PPM-2}. On both G1\_PPM and G2\_PPM, all three methods exactly recover the planted community structure, as indicated by $\epsilon_{intra}=\epsilon_{inter}=0$ and ARI=$1$. Although BM requires the least computational time on both instances, TSCM also achieves exact recovery with a substantially lower computational cost than Gurobi. 	
	The difference among the methods becomes more evident on G3\_PPM. Both TSCM and Gurobi exactly recover the ground-truth partition, yielding $\epsilon_{intra}=\epsilon_{inter}=0$ and ARI=$1$. In contrast, BM yields $\epsilon_{intra}=0.0074$, $\epsilon_{inter}=0.0018$, and ARI=$0.9800$, indicating a small number of misclassified vertices. Although BM remains the fastest method, TSCM achieves exact recovery with a total computation time of only $0.23$s, which is also slightly lower than that of Gurobi. It is also worth noting that the clustering results produced by all three methods on all three PPM instances satisfy the clustering-quality condition in \eqref{cq-standard}, indicating that the resulting partitions consistently preserve the expected separation between intra-cluster and inter-cluster connectivity.
	
	An interesting observation can be made by jointly considering Tables~\ref{tab:compare-PPM-1} and~\ref{tab:compare-PPM-2}. In particular, on G3\_PPM, BM obtains a smaller Stage~1 objective value than TSCM, whereas TSCM produces an exact final partition and BM does not. This demonstrates that the objective value of the medoid-selection problem and the final clustering accuracy are not necessarily monotonically related. Consequently, it is important to evaluate the proposed framework from both the optimization and clustering perspectives, rather than relying solely on the objective value of the medoid-selection model.
	
	\begin{remark}
		We note that changing the value of $\alpha$ directly alters the objective function itself; hence, objective values obtained under different $\alpha$ settings are not directly comparable. Specifically, under the common parameter setting $\alpha=\alpha_{0}=(2n/K-1)\beta$, the BM method yields an objective value of $-1.8038$, with corresponding clustering metrics $\epsilon_{intra}=0.0074$, $\epsilon_{inter}=0.0018$, and {\rm ARI=}$0.9800$. When $\alpha$ is increased to $1.5\alpha_{0}$, BM produces better clustering results, namely $\epsilon_{intra}=\epsilon_{inter}=0$ and {\rm ARI=}$1$, and the corresponding objective value becomes $-13.0036$. It should be emphasized that this value cannot be directly compared with the objective values and clustering metrics reported in Table~\ref{tab:compare-PPM-1} and Table~\ref{tab:compare-PPM-2}, as they are computed under different $\alpha$ values. Therefore, the results obtained with $1.5\alpha_{0}$
		are used solely to illustrate the sensitivity of clustering performance to the model parameter and are not included in the main optimization comparison. Moreover, BM requires a larger $\alpha$ to reach optimality, which also indirectly reflects the robustness of our method.
	\end{remark}
	
	\subsubsection{Results on SBM Graph Models}
	In this section, we extend our evaluation to SBM graph models. Compared to the PPM-generated graphs, SBM graphs introduce more structured community patterns with heterogeneous intra- and inter-community edge probabilities, allowing us to further examine the optimizers' adaptability to varying levels of community clarity and structural complexity. For SBM Graph Models, we set $\beta=K/n$, $\alpha=(2n/K-1)\beta$. 
	
	\paragraph{Optimization Performance of Medoid Selection}
	
	For the primary comparison on the SBM instances, TSCM and BM are terminated according to their respective stopping criteria, while Gurobi is assigned a time limit of \(100\)s. This time limit is chosen because BM typically requires approximately \(100\)s to complete, whereas TSCM usually terminates substantially earlier. Thus, the \(100\)s limit provides a common upper computational budget for evaluating the quality of the medoid-selection solutions. For each method, we report the objective value \(f(\mathbf{x})\) of the original medoid-selection model together with the actual Stage~1 computation time.
	
		\begin{table}[ht]
		\centering
		\caption{Comparison of optimization performance on SBM models with $n=3000$ and $K=30$.}
		\label{tab:compare-sbm-1}
		\setlength{\tabcolsep}{4.5pt} 
			\begin{tabular*}{0.9\textwidth}{@{\extracolsep{\fill}}ccccccc}
				\toprule
				\multirow{2.4}{*}{method} & \multicolumn{2}{c}{G1\_SBM} & \multicolumn{2}{c}{G2\_SBM} & \multicolumn{2}{c}{G3\_SBM} \\
				\cmidrule(r){2-3} \cmidrule(r){4-5} \cmidrule(r){6-7} 
				& $f(\bf{x})$& CPU time(s)     & $f(\bf{x})$ & CPU time(s) &  $f(\bf{x})$ & CPU time(s) \\
				\midrule
				TSCM &-9.8059 &  8.22  & -6.5394 & 6.16 &  -5.5534 & 5.49\\
				\midrule
				Gurobi & -7.9869 & 102.07 & -8.1178 & 102.74 & -8.9117 & 110.10 \\
				\midrule
				BM & -12.4853 & 81.53  &-11.2006 & 82.55  & -10.9057 & 82.08\\
				\bottomrule
			\end{tabular*}%
		%
		\centering
		\setlength{\tabcolsep}{4.5pt} 
			\begin{tabular*}{0.9\textwidth}{@{\extracolsep{\fill}}ccccccc}
				\toprule
				\multirow{2.4}{*}{method} & \multicolumn{2}{c}{G4\_SBM} & \multicolumn{2}{c}{G5\_SBM} & \multicolumn{2}{c}{G6\_SBM} \\
				\cmidrule(r){2-3} \cmidrule(r){4-5} \cmidrule(r){6-7} 
				& $f(\bf{x})$& CPU time(s)     & $f(\bf{x})$ & CPU time(s) &  $f(\bf{x})$ & CPU time(s) \\
				\midrule
				TSCM & -9.5523 &  6.44  & -7.0040 &  5.00 & -4.1439 & 4.83\\
				\midrule
				Gurobi & -5.5601 & 108.96  & -6.5012 & 101.68 & -8.5019 & 103.37\\
				\midrule
				BM & -11.6005 & 81.41  & -10.1662 &  82.59   & -10.2101 & 81.39 \\
				\bottomrule
			\end{tabular*}%
		%
		\centering
		\setlength{\tabcolsep}{4.5pt} 
			\begin{tabular*}{0.9\textwidth}{@{\extracolsep{\fill}}ccccccc}
				\toprule
				\multirow{2.4}{*}{method} & \multicolumn{2}{c}{G7\_SBM} & \multicolumn{2}{c}{G8\_SBM} & \multicolumn{2}{c}{G9\_SBM} \\
				\cmidrule(r){2-3} \cmidrule(r){4-5} \cmidrule(r){6-7} 
				& $f(\bf{x})$& CPU time(s)     & $f(\bf{x})$ & CPU time(s) &  $f(\bf{x})$ & CPU time(s) \\
				\midrule
				TSCM & -8.1797 & 5.40 & -5.8455 & 4.55 & -4.5613 & 5.16 \\
				\midrule
				Gurobi & -7.1344 & 100.82  & -6.8090 & 109.99 & -7.8358 & 101.56 \\
				\midrule
				BM & -11.2221 & 82.62 & -9.6284 & 81.85  & -9.6435 & 82.68 \\
				\bottomrule
			\end{tabular*}%
		%
	\end{table}%
	
	Table~\ref{tab:compare-sbm-1} reports the optimization performance of the three methods on the nine SBM instances. BM attains the smallest objective value on all nine instances, whereas TSCM and Gurobi exhibit instance-dependent behavior. In particular, TSCM obtains a smaller objective value than Gurobi on G1\_SBM, G4\_SBM, G5\_SBM, and G7\_SBM, while Gurobi performs better in terms of the objective value on the remaining instances.
	
	The main advantage of TSCM at the medoid-selection stage lies in its computational efficiency. Its average Stage~1 computation time over the nine instances is approximately $5.70$s, compared with $82.08$s for BM and $104.59$s for Gurobi. Thus, TSCM is approximately $14.4$ times faster than BM and $18.4$ times faster than Gurobi on average. These results indicate that, although TSCM does not necessarily minimize the Stage~1 objective as aggressively as BM, it can identify feasible medoid sets at a substantially lower computational cost. The practical quality of these medoid sets is further examined through the unified assignment procedure below.
	
	\paragraph{Clustering Quality after Unified Assignment}
	
	Similarly, to further assess the practical effect of the selected medoids, the solutions obtained by TSCM, BM, and Gurobi under the above computational protocol are passed through the same nearest-medoid assignment procedure. We report the resulting ARI, the intra-cluster and inter-cluster density errors, and the total computation time, including both medoid selection and the unified assignment stage. Since the assignment rule is identical for all methods, differences in the final clustering quality can be attributed primarily to the medoid sets identified in Stage~1.

			\begin{table}[ht]
			\centering
			\caption{Comparison of some clustering metrics on SBM models with $n = 3000$ and $K = 30$.}
			\label{tab:compare-SBM-2}%
			\setlength{\tabcolsep}{3.5pt} 
				\begin{tabular*}{0.9\textwidth}{@{\extracolsep{\fill}}cccccccccc}
					\toprule
					\multirow{2.4}{*}{method} & \multicolumn{3}{c}{G1\_SBM} & \multicolumn{3}{c}{G2\_SBM} & \multicolumn{3}{c}{G3\_SBM} \\
					\cmidrule(r){2-4} \cmidrule(r){5-7} \cmidrule(r){8-10} 
					& $\epsilon_{intra}$ & $\epsilon_{inter}$  & ${\rm T_{total}}$(s)     & $\epsilon_{intra}$ & $\epsilon_{inter}$  & ${\rm T_{total}}$(s) & $\epsilon_{intra}$ & $\epsilon_{inter}$  & ${\rm T_{total}}$(s) \\
					\midrule
					TSCM & 0.0022 & 0.0002 & 8.31  & 0.0644 & 0.0064  & 6.20 & 0.0822 & 0.0075  & 5.51  \\
					\midrule
					Gurobi & 0.6522 & 0.0320 & 103.08 & 0.5531 & 0.0267  & 104.66 & 0.6055 & 0.0290 & 111.26 \\
					\midrule
					BM & 0  & 0.0003 & 81.55 & 0.0050 & 0.0005 & 82.58  & 0.0595 & 0.0038  & 82.56\\
					\bottomrule
				\end{tabular*}%
				\centering
				\setlength{\tabcolsep}{3.5pt} 
			\begin{tabular*}{0.9\textwidth}{@{\extracolsep{\fill}}cccccccccc}
				\toprule
				\multirow{2.4}{*}{method} & \multicolumn{3}{c}{G4\_SBM} & \multicolumn{3}{c}{G5\_SBM} & \multicolumn{3}{c}{G6\_SBM} \\
				\cmidrule(r){2-4} \cmidrule(r){5-7} \cmidrule(r){8-10} 
				& $\epsilon_{intra}$ & $\epsilon_{inter}$  & ${\rm T_{total}}$(s)     & $\epsilon_{intra}$ & $\epsilon_{inter}$  & ${\rm T_{total}}$(s) & $\epsilon_{intra}$ & $\epsilon_{inter}$  & ${\rm T_{total}}$(s)  \\
				\midrule
				TSCM & 0.0014 & 0.0002 & 6.46  & 0.0042 & 0.0002 & 5.04 & 0.0623 & 0.0072 & 5.06 \\
				\midrule
				Gurobi & 0.6467 & 0.0309 & 111.71 & 0.6257 & 0.0277 & 104.84 & 0.5884 & 0.0286 & 105.69 \\
				\midrule
				BM & 0.0603 & 0.0035 &  81.62  & 0.0057 & 0.0003 & 82.82 &  0.1949 & 0.0082 & 81.43 \\
				\bottomrule
			\end{tabular*}%
			\centering
			\setlength{\tabcolsep}{3.5pt} 
			\begin{tabular*}{0.9\textwidth}{@{\extracolsep{\fill}}cccccccccc}
				\toprule
				\multirow{2.4}{*}{method} & \multicolumn{3}{c}{G7\_SBM} & \multicolumn{3}{c}{G8\_SBM} & \multicolumn{3}{c}{G9\_SBM} \\
				\cmidrule(r){2-4} \cmidrule(r){5-7} \cmidrule(r){8-10} 
				& $\epsilon_{intra}$ & $\epsilon_{inter}$  & ${\rm T_{total}}$(s)     & $\epsilon_{intra}$ & $\epsilon_{inter}$  & ${\rm T_{total}}$(s) & $\epsilon_{intra}$ & $\epsilon_{inter}$  & ${\rm T_{total}}$(s) \\
				\midrule
				TSCM & 0.0018 & 0 & 5.43 & 0.0016 & 0.0003 & 4.58  & 0.0311 & 0.0021 & 5.23  \\
				\midrule
				Gurobi & 0.5187  & 0.0259 & 102.79 & 0.4829 & 0.0225 & 110.49 & 0.4907 & 0.0232 & 104.29 \\
				\midrule
				BM & 0.0316  & 0.0026 & 82.65 & 0.0497 & 0.0025 & 82.26 & 0.0623 & 0.0028& 82.73  \\
				\bottomrule
			\end{tabular*}%
		\end{table}%
		
		Table~\ref{tab:compare-SBM-2} reports the density-based clustering errors and the total computation time after applying the same nearest-medoid assignment procedure to the medoids identified by the three methods. The corresponding ARI values are listed separately in Table~\ref{tab:compare-SBM-ARI} and visualized in Fig.~\ref{fig:compare_ARI_linegraph}.
		
		The density-based errors in Table~\ref{tab:compare-SBM-2} first demonstrate the quality of the partitions obtained after the unified nearest-medoid assignment. Compared with BM, TSCM simultaneously yields smaller $\epsilon_{intra}$ and $\epsilon_{inter}$ on six of the nine SBM instances, namely G4\_SBM -- G9\_SBM. BM performs better on G2\_SBM and G3\_SBM, while the two methods exhibit mixed behavior on G1\_SBM. In contrast, Gurobi produces substantially larger density deviations on all nine instances. Overall, the density-based results indicate that the medoids identified by TSCM generally lead to partitions that preserve the planted intra-cluster and inter-cluster connectivity more accurately than those obtained by the competing methods. Moreover, all clustering results produced by the three methods on the nine SBM instances satisfy the clustering-quality condition in~\eqref{cq-standard}.
		
		A particularly clear advantage of TSCM is its computational efficiency. Its total computation time, including both medoid selection and the unified assignment stage, ranges from only $4.58$s to $8.31$s over the nine SBM instances. By comparison, BM requires approximately $81.43$--$82.82$s, whereas Gurobi requires approximately $102.79$--$111.71$s under the prescribed computational protocol. Hence, TSCM is consistently the fastest method on all nine instances, while still producing high-quality graph partitions according to the density-based criteria.

		\begin{table}[ht]
			\centering
			\caption{Detailed ARI Values of Different Methods on Each SBM Graph.}
			\label{tab:compare-SBM-ARI}%
			\setlength{\tabcolsep}{4.2pt} 
				\begin{tabular*}{0.95\textwidth}{@{\extracolsep{\fill}}cccccccccc}
					\toprule
					Method & G1\_SBM & G2\_SBM & G3\_SBM& G4\_SBM & G5\_SBM & G6\_SBM& G7\_SBM & G8\_SBM & G9\_SBM \\
				
					\midrule
					TSCM &  1  & 0.9155 & 0.9163 & 1& 0.9989& 0.8947& 1& 0.9988& 0.9486  \\
					\midrule
					Gurobi &  0.3709 & 0.4618 & 0.2791 & 0.2046& 0.2603& 0.2201& 0.4214& 0.3757& 0.2846 \\
					\midrule
					BM &  0.9995 & 0.9917 & 0.9452 & 0.9721& 0.9957& 0.8339& 0.9829& 0.9622& 0.9373\\
					\bottomrule
				\end{tabular*}%
					\end{table}%
					
		The ARI results reported in Table~\ref{tab:compare-SBM-ARI} further confirm the clustering effectiveness of TSCM. TSCM achieves a higher ARI than BM on seven of the nine graphs, namely G1\_SBM and G4\_SBM--G9\_SBM, whereas BM performs better on G2\_SBM and G3\_SBM. Moreover, TSCM outperforms Gurobi on all nine instances. In particular, TSCM exactly recovers the planted partitions on G1\_SBM, G4\_SBM, and G7\_SBM, where $\mathrm{ARI}=1$, and achieves nearly exact recovery on G5\_SBM and G8\_SBM, with ARI values of $0.9989$ and $0.9988$, respectively. Across the nine instances, the ARI values obtained by TSCM range from $0.8947$ to $1$, compared with $0.8339$--$0.9995$ for BM and $0.2046$--$0.4618$ for Gurobi.

			\begin{figure}[htp]
			\centering
			\includegraphics[width=1\textwidth]{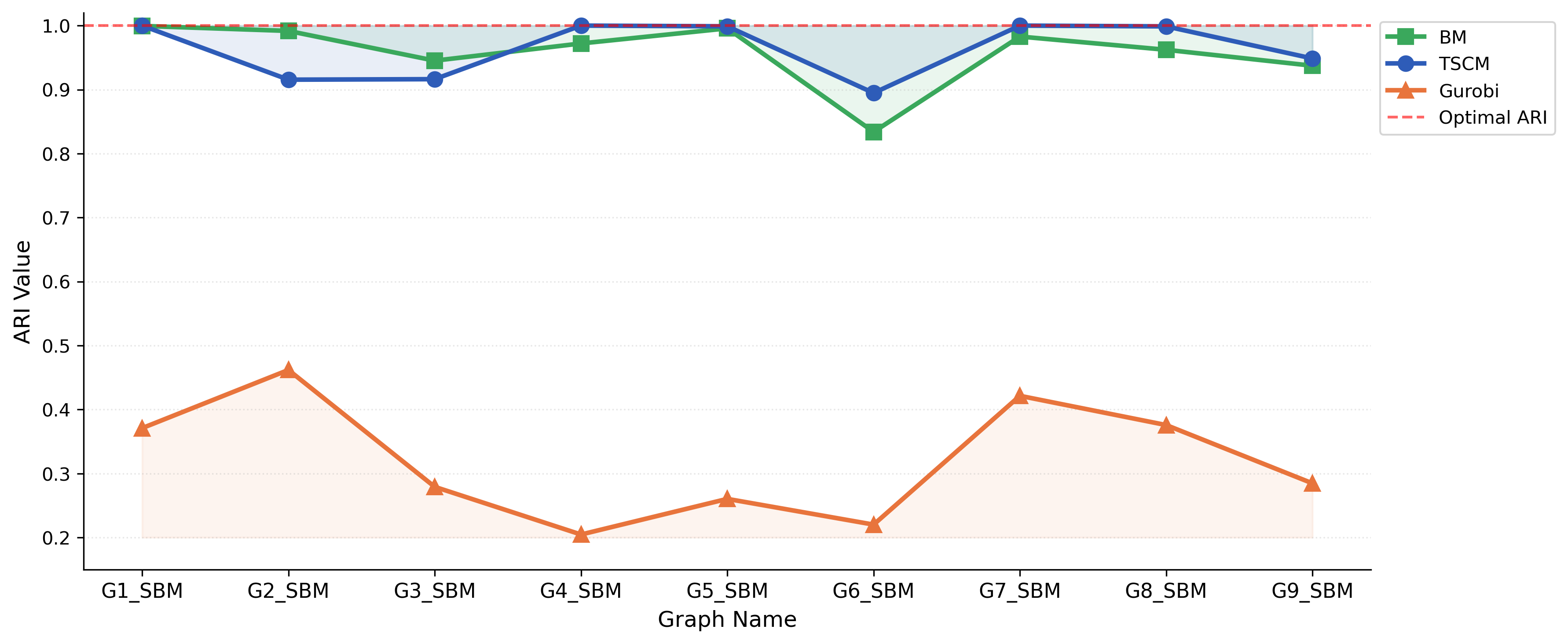}\\
			\caption{Comparison of ARI values obtained by different methods on nine SBM instances. The red dashed line indicates the optimal ARI value of $1$.}
			\label{fig:compare_ARI_linegraph}
		\end{figure}
		
		Figure~\ref{fig:compare_ARI_linegraph} provides a visual comparison of the ARI values obtained by the three methods on the nine SBM instances. The red dashed horizontal line represents the optimal ARI value of $1$, which corresponds to exact agreement between the obtained clustering partition and the ground-truth partition. Therefore, a curve closer to the red dashed line indicates better clustering performance. The blue curve represents TSCM, the green curve represents BM, and the orange curve represents Gurobi. It can be observed that the TSCM curve stays very close to the optimal ARI line over all nine SBM instances and reaches the upper bound of $1$ on G1\_SBM, G4\_SBM, and G7\_SBM. The BM curve also remains close to the optimal line on most instances, but lies below the TSCM curve on seven of the nine SBM graphs. In contrast, the Gurobi curve is consistently farther from the optimal ARI line, indicating a substantially larger discrepancy between its resulting partitions and the ground-truth community structures.
		
		These results demonstrate that the objective value of the medoid-selection model and the quality of the final graph partition are not necessarily monotonically related. The Stage~1 objective evaluates the selected medoid set according to the proposed optimization model, whereas the final clustering quality additionally depends on how these medoids induce the partition through the assignment stage. Consequently, the performance of the proposed two-stage framework should be assessed from both optimization and clustering perspectives. From this end-to-end viewpoint, TSCM provides the best overall balance between clustering accuracy and computational efficiency on the considered SBM instances.

		\begin{remark}[Extended-time behavior of Gurobi]
			The $100$s time limit used in the primary experiments is intended to provide a comparable computational budget. To further investigate the effect of a substantially larger budget, we additionally allow Gurobi to run for up to $500$s on the SBM instances. The extended runs improve the feasible solutions and, on several instances, lead to noticeably higher ARI values after the unified assignment procedure. This indicates that the clustering performance of Gurobi can benefit from additional computation. Nevertheless, these improvements are obtained at a substantially higher computational cost than TSCM, which typically terminates within approximately $5$--$9$s on the same instances. The complete extended-time results are reported in Appendix~\ref{app:gurobi-extended}.
		\end{remark}

	\subsection{Real-world graphs}
	In this section, we evaluate the proposed method on two widely used real-world community detection benchmarks: Zachary's Karate Club network (Karate\_club) and the United States College Football Division IA network during the 2000 season (Football). These networks contain manually identified community structures, which provide reference partitions for evaluating clustering performance. The Karate Club network represents social interactions among members of a university karate club, while the Football network describes game interactions among college football teams, where conferences naturally form communities. The datasets are publicly available from the CommunityDetection repository: \url{https://github.com/qinyuenlp/CommunityDetection/tree/master/data}.
	The basic characteristics of these networks are summarized in Table~\ref{real-graph-c}.
	
	\begin{table}[ht]
		\centering
		\caption{Graph characteristics of two real-world graphs.}
		\label{real-graph-c}%
		\begin{tabular}{ccccc}
			\toprule
			\multirow{2}[4]{*}{Graph name} & \multicolumn{4}{c}{Graph characteristics} \\
			\cmidrule{2-5}          & $n$   & $|E|$   & $\kappa$ & $K$  \\
			\midrule
			Karate\_club & 34    & 78    & 0.14  & 2 \\
			\midrule
			Football & 115   & 613   & 0.09  & 12 \\
			\bottomrule
		\end{tabular}%
		
	\end{table}%
	
	For real-world networks, the underlying graph generation mechanism is unknown, and therefore the density-based evaluation criteria used for synthetic graphs cannot be directly applied. Instead, we evaluate the clustering results by comparing the obtained partitions with the available reference community assignments. We set $\beta=K/n$, $\alpha=1.5\times(2n/K-1)\beta$.
	
	After obtaining the Medoid sets from different methods, the same nearest-Medoid assignment strategy is applied to generate the final community partitions. Since community labels are invariant to permutations, the predicted labels are first aligned with the reference labels using an optimal matching strategy. We then count the number of incorrectly assigned vertices between the obtained partition and the ground-truth partition. A smaller number of misclassified nodes indicates better agreement with the known community structure.
	
	In addition, the computational time of all compared methods on these real-world networks is less than one second. Therefore, the running time difference is negligible and the comparison mainly focuses on the quality of the obtained community partitions.
	
	\subsubsection{Results on Karate Club}
	
	This subsection considers the simplified unweighted version of Zachary's Karate Club network \cite{no3}. The network contains $34$ vertices, where each vertex represents a club member and each edge represents a social relationship between two members. Due to an internal conflict, the club eventually split into two factions, and the ground-truth partition is shown in Figure~\ref{fig:atrue}. In each subfigure of Figure~\ref{fig:1234}, nodes with the same color belong to the same cluster; however, the same color in different subfigures does not imply any correspondence across the clustering results.
	
	We apply the proposed TSCM method to this network to identify its two communities, and compare its clustering results with those obtained by Gurobi and BM under the same CPU runtime of about $1$s. Previous research \cite{miasnikof2024graph} has shown that Gurobi and BM can serve as effective comparison methods for graph clustering problems. To provide a more intuitive comparison of the clustering results, Figure~\ref{fig:1234} shows the ground-truth partition as well as the partitions obtained by TSCM, Gurobi, and BM, where misclassified nodes are marked with red boxes.
	
	As shown in Figure~\ref{fig:aTSCM}, TSCM misclassifies only $1$ node, whereas both Gurobi and BM misclassify $3$ nodes, as shown in Figure~\ref{fig:agurobi} and Figure~\ref{fig:aBM}. Therefore, on the Karate Club network, TSCM produces fewer misclassified nodes than Gurobi and BM.

	\begin{figure}[ht]
		\centering
		\begin{minipage}{1\linewidth}	
			\subfigure[The true clustering]{
				\label{fig:atrue}
				\includegraphics[width=0.47\linewidth]{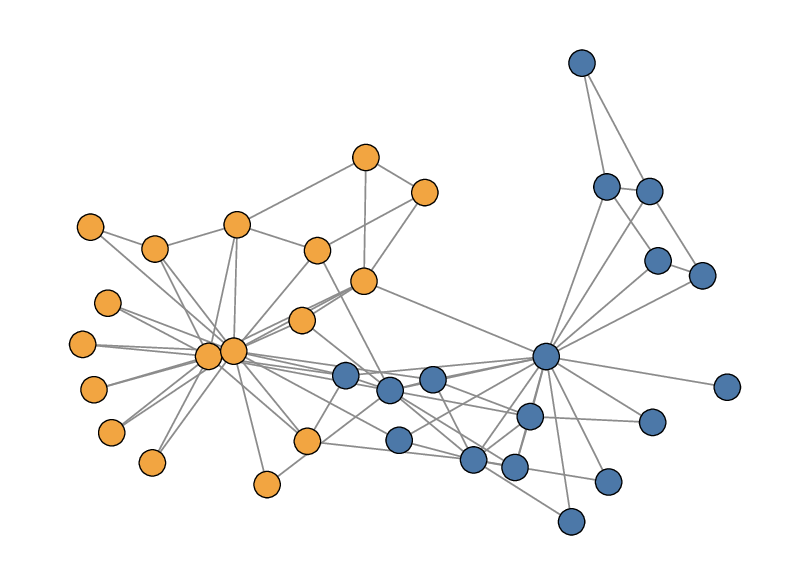}	
			}\noindent
			\subfigure[The clustering result from TSCM ]{
				\label{fig:aTSCM}
				\includegraphics[width=0.47\linewidth]{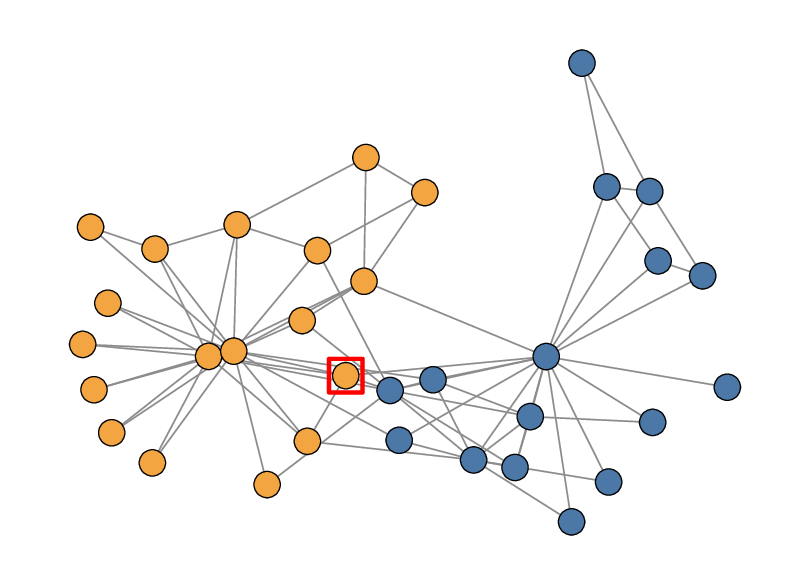}
			}
		\end{minipage}
		\vskip -0.2cm 
		\begin{minipage}{1\linewidth }
			\subfigure[The clustering result from Gurobi]{
				\label{fig:agurobi}
				\includegraphics[width=0.47\linewidth]{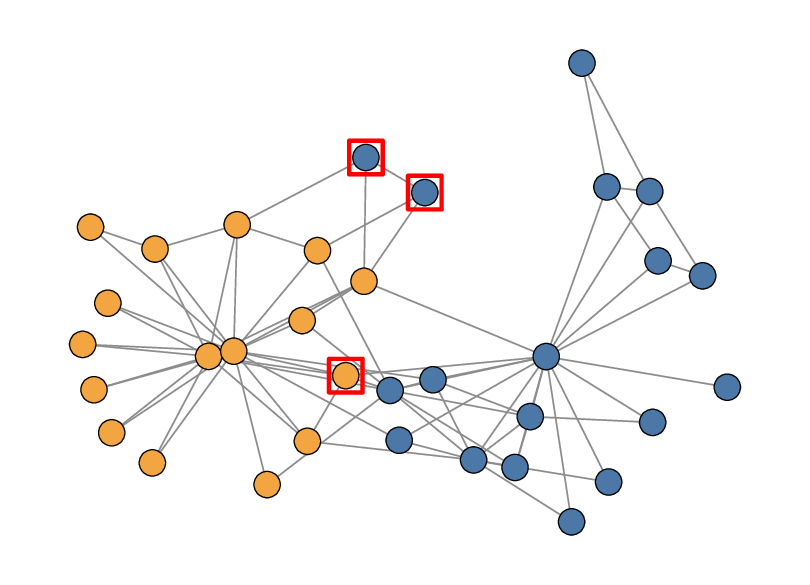}
			}\noindent
			\subfigure[The clustering result from BM]{
				\label{fig:aBM}
				\includegraphics[width=0.47\linewidth]{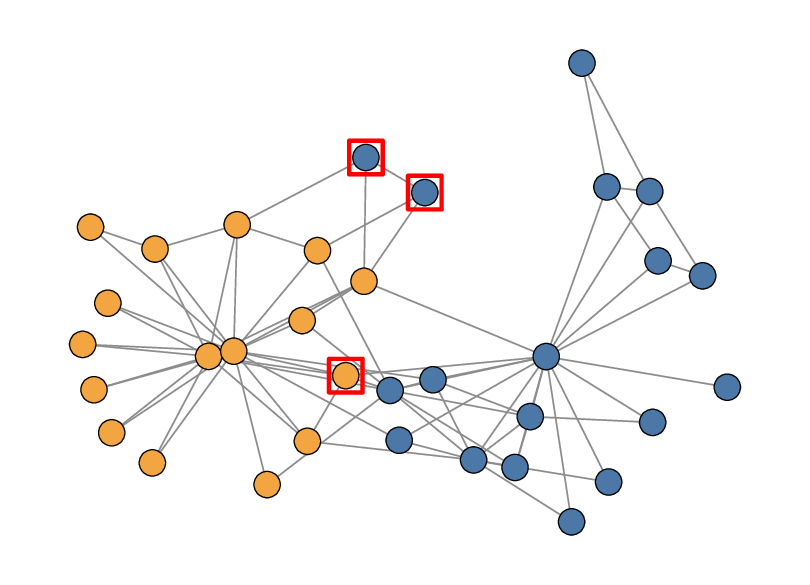}
			}
		\end{minipage}
		\vspace{-0.11in}	
		\caption{Clustering results on the friendship network of Zachary's karate club. Nodes of different colors belong to different factions, and red squares mark misclassified nodes}
		\label{fig:1234}
	\end{figure}
	
	\subsubsection{Results on Football}	
	This section uses the American college football IA group 2000 season network \cite{no3}. The network contains $115$ teams, where each vertex represents a team and an edge indicates that two teams played at least one game against each other during the regular season. According to the conferences to which the teams belong, the network has $12$ ground-truth communities. Since teams within the same conference tend to play each other more frequently, they generally have more connections within the same community. The ground-truth partition of this network is shown in Figure~\ref{fig:2234}, where different colors indicate different communities.
	
	Under the same CPU runtime of about $1$s, we compare the clustering results of TSCM with obtained by Gurobi and BM. The corresponding visualization results are shown in Figure~\ref{fig:bTSCM}, Figure~\ref{fig:bgurobi}, and Figure~\ref{fig:bBM}, respectively, where the misclassified nodes are highlighted by red squares.

			\begin{figure}[ht]
		\centering
		\begin{minipage}{1\linewidth}	
			\centering
			\subfigure[The true clustering]{
				\label{fig:btrue}
				\includegraphics[width=0.40\linewidth]{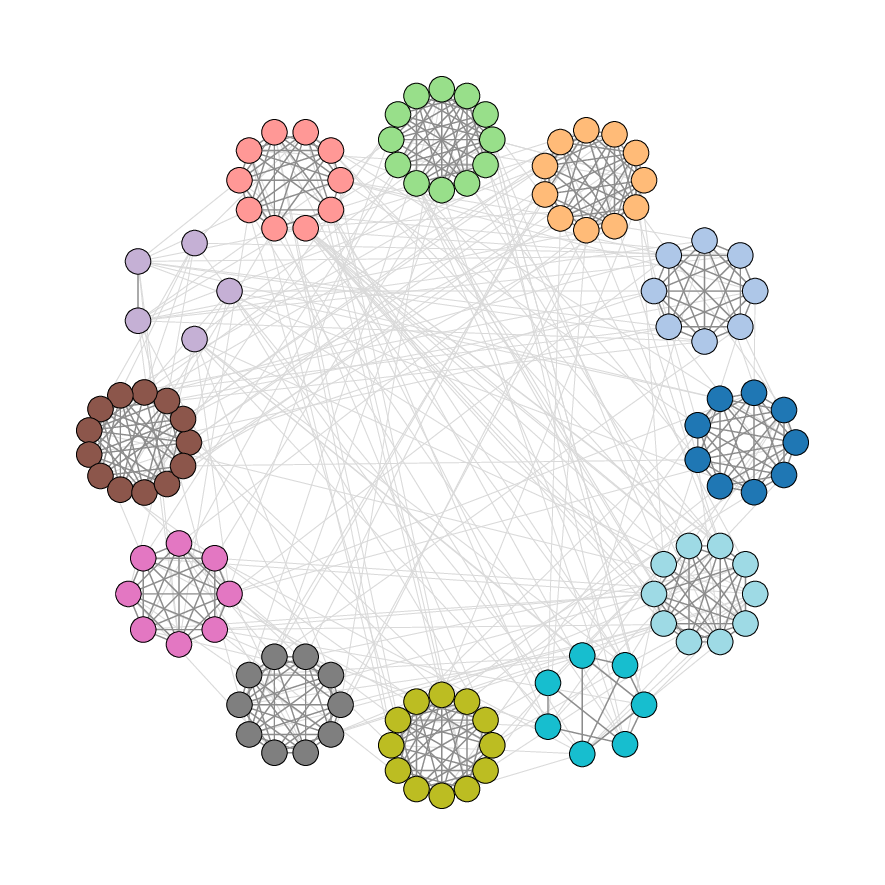}	
			}\noindent\hspace{0.3in}
			\subfigure[The clustering result from TSCM ]{
				\label{fig:bTSCM}
				\includegraphics[width=0.40\linewidth]{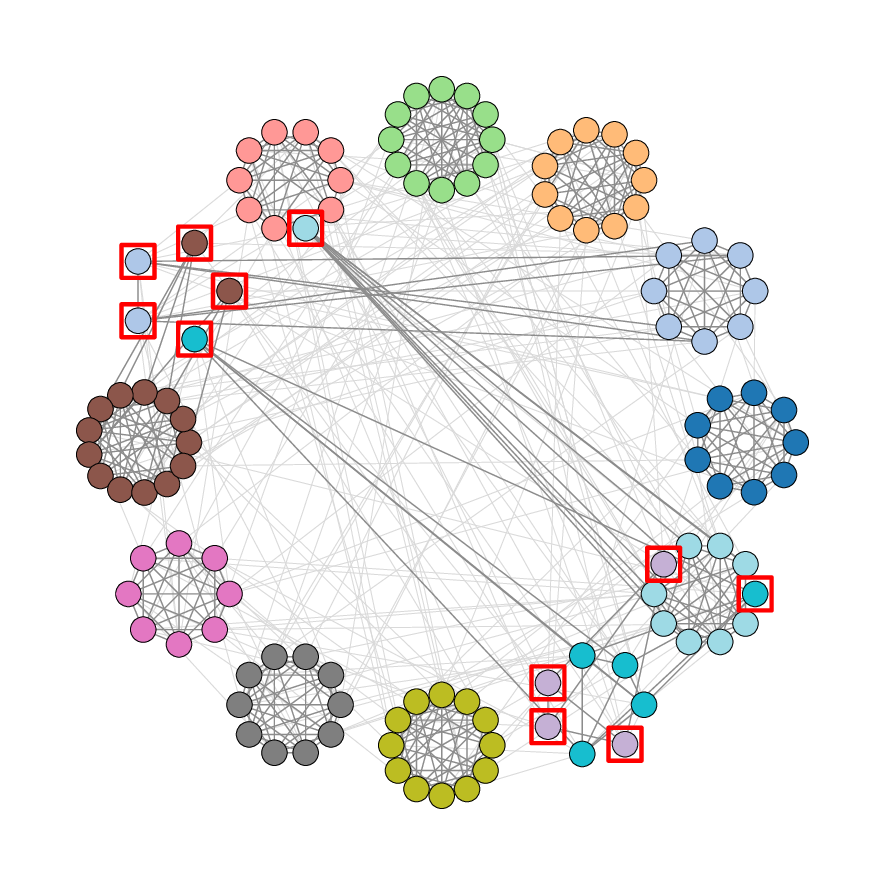}
			}
		\end{minipage}
		\vskip -0.3cm 
		\begin{minipage}{1\linewidth }
			\centering
			\subfigure[The clustering result from Gurobi]{
				\label{fig:bgurobi}
				\includegraphics[width=0.40\linewidth]{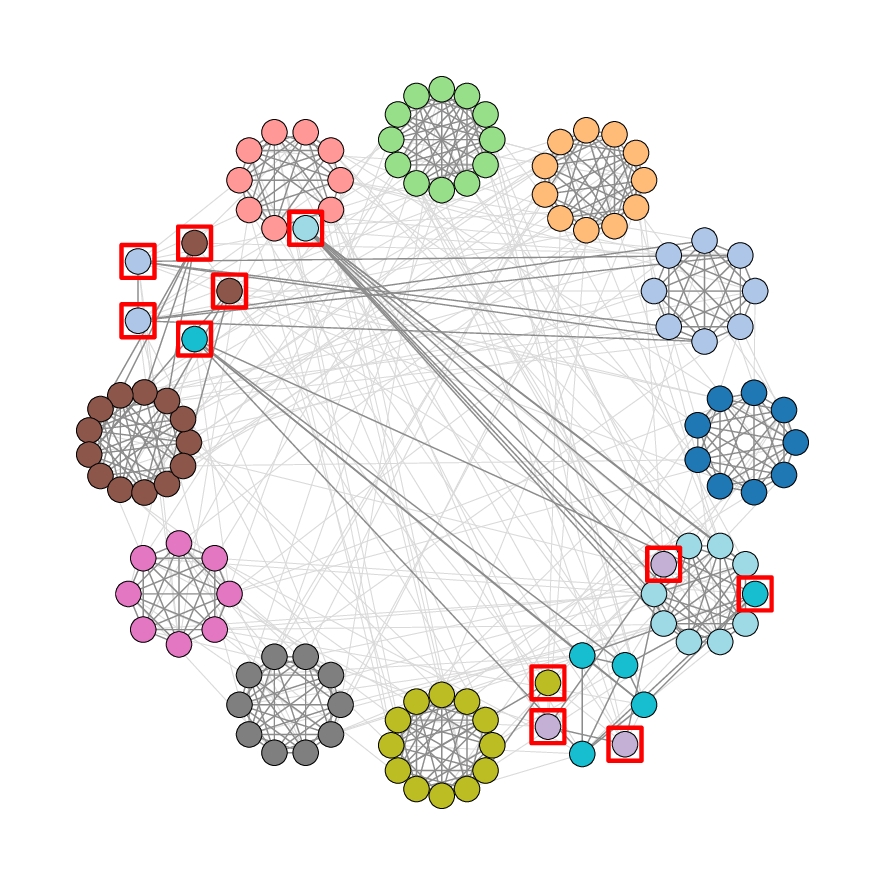}
			}\noindent\hspace{0.3in}
			\subfigure[The clustering result from BM]{
				\label{fig:bBM}
				\includegraphics[width=0.40\linewidth]{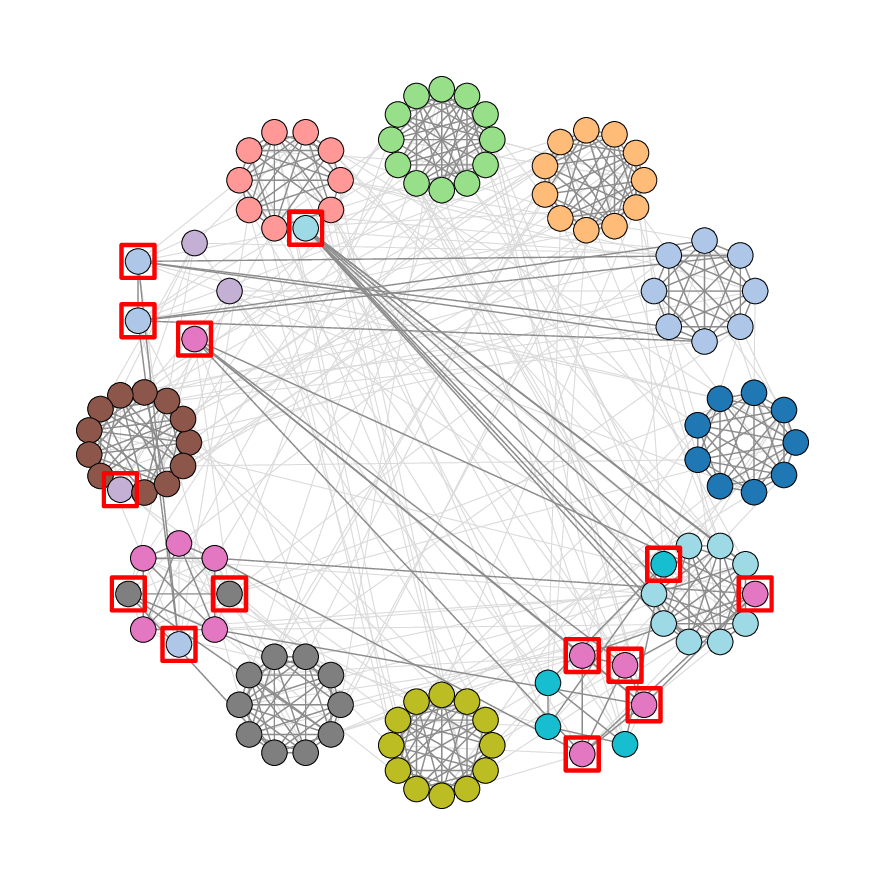}
			}
		\end{minipage}
		\vspace{-0.11in}	
		\caption{Clustering results on the US College Football Division IA 2000 season graph. Each circle groups nodes (teams) belonging to the same cluster, and red squares mark misclassified nodes}
		\label{fig:2234}
	\end{figure}
		
	As can be seen from the figures, both TSCM and Gurobi misclassify $11$ nodes, while BM misclassifies $14$ nodes. Therefore, on the Football network, TSCM and Gurobi produce the same number of misclassified nodes, and both yield fewer misclassifications than BM.
	
	Combining the experimental results on the two real-world networks, Karate Club and Football, it can be seen that under the same computational time budget, TSCM yields fewer misclassified nodes than Gurobi and BM on the Karate Club network; on the Football network, TSCM obtains the same number of misclassified nodes as Gurobi and fewer than BM. These results indicate that TSCM can be applied to community identification in real-world networks and achieve competitive clustering results.
	
	To conclude, we find that TSCM demonstrate the ability to recover clusters not only in synthetic graphs but also in real-world graphs.
	\section{Conclusion}\label{sec:conclusion}

	In this paper, we investigated the quadratic $k$-medoids graph clustering problem from the perspective of CM optimization and proposed a two-stage clustering method (TSCM). By identifying the binary feasible set of the quadratic $k$-medoids model as a discrete CM set, we established a connection between graph clustering and CM optimization. The original binary problem was then reformulated over the convex hull of its feasible set, and a negative squared-norm penalty was introduced to promote extreme-point solutions. This reformulation enables the discrete medoid-selection problem to be addressed through continuous optimization over the convex hull.
	
	Building on this formulation, we developed TSCM to identify representative medoids and construct the final graph partition. In Stage~1, a penalty-continuation projected gradient (PCPG) solver was employed to approximately solve the penalized medoid-selection problem. Starting from a refined multi-start initialization, PCPG progressively increases the penalty parameter and employs a BB projected gradient method combined with a projected Armijo line search to solve each fixed-penalty subproblem. Binary candidate solutions are retained along the continuation path according to the prescribed selection criterion. In Stage~2, the selected medoids are converted into a graph partition through a unified assignment procedure. Specifically, each vertex is assigned to exactly one cluster using a two-level distance criterion, with the Jaccard distance serving as the primary measure and the shortest-path distance providing a secondary criterion when the Jaccard distance cannot distinguish between candidate medoids.
	
	Numerical experiments on PPM and SBM graphs demonstrate the effectiveness and computational efficiency of the proposed method. On the PPM instances, TSCM exactly recovers the planted partitions while maintaining a favorable balance between solution quality and computational cost. On the larger SBM instances, TSCM exhibits substantial computational advantages over Gurobi and BM. In terms of clustering accuracy, TSCM achieves higher ARI values than BM on seven of the nine SBM instances and outperforms Gurobi on all nine instances. Moreover, it exactly recovers the planted partitions on three SBM instances and achieves nearly exact recovery on two additional instances, while requiring substantially less computational time than the competing methods. Finally, experiments on two real-world graph datasets further validate the efficiency and high accuracy of the proposed method.
	
	The numerical results also reveal that a smaller objective value in the medoid-selection problem does not necessarily lead to a more accurate final graph partition. This finding highlights the importance of evaluating graph clustering methods from both the optimization and clustering perspectives, rather than relying solely on the objective value of the medoid-selection model. Overall, the proposed framework provides an effective and computationally efficient approach to graph clustering by combining continuous optimization for medoid selection with a graph assignment procedure. Future work will investigate more effective penalty-parameter and initialization strategies, as well as the relationship between medoid-selection objectives and the quality of the resulting graph partitions.
	

\section*{Acknowledgments}
This work is supported by National Natural Science Foundation
	of China (NSFC) No. 12271526.
      
	\bibliographystyle{unsrt}
	\bibliography{ref}

@article{cq-1,
	title={A density-based statistical analysis of graph clustering algorithm performance},
	author={Miasnikof, Pierre and Shestopaloff, Alexander Y and Bonner, Anthony J and Lawryshyn, Yuri and Pardalos, Panos M},
	journal={Journal of Complex Networks},
	volume={8},
	number={3},
	pages={cnaa012},
	year={2020},
	publisher={Oxford University Press}
}

@inproceedings{cq-2,
	title={A statistical performance analysis of graph clustering algorithms},
	author={Miasnikof, Pierre and Shestopaloff, Alexander Y and Bonner, Anthony J and Lawryshyn, Yuri},
	booktitle={International Workshop on Algorithms and Models for the Web-Graph},
	pages={170-184},
	year={2018},
	organization={Springer}
}

@article{cq-3,
	title={Analysis of weighted networks},
	author={Newman, Mark EJ},
	journal={Physical Review E-Statistical, Nonlinear, and Soft Matter Physics},
	volume={70},
	number={5},
	pages={056131},
	year={2004},
	publisher={APS}
}

@article{Teng_2025, 
	title={An efficient quadratic penalty method for a class of graph clustering problems}, 
	author={Teng, Wenshun and Li, Qingna},
	url={http://dx.doi.org/10.1007/s11081-025-10042-9}, 
	DOI={10.1007/s11081-025-10042-9}, 
	journal={Optimization and Engineering}, 
	publisher={Springe}, 
	year = {2025},
	pages={1-29}
	}

@article{jaccard-1,
	title={{\'E}tude comparative de la distribution florale dans une portion des Alpes et des Jura},
	author={Jaccard, Paul},
	journal={Bulletin de la Soci{\'e}t{\'e} Vaudoise des Sciences Naturelles},
	volume={37},
	pages={547-579},
	year={1901}
}

@book{jaccard-2,
	title={The extended Jaccard distance in complex networks},
	author={Camby, {\'E}glantine and Caporossi, Gilles},
	year={2017},
	publisher={GERAD, {\'E}cole des hautes {\'e}tudes commerciales}
}

@article{jaccard-3,
	title={An empirical comparison of connectivity-based distances on a graph and their computational scalability},
	author={Miasnikof, Pierre and Shestopaloff, Alexander Y and Pitsoulis, Leonidas and Ponomarenko, Alexander},
	journal={Journal of Complex Networks},
	volume={10},
	number={1},
	pages={cnac003},
	year={2022},
	publisher={Oxford University Press}
}

@inproceedings{jaccard-4,
	title={Distances on a graph},
	author={Miasnikof, Pierre and Shestopaloff, Alexander Y and Pitsoulis, Leonidas and Ponomarenko, Alexander and Lawryshyn, Yuri},
	booktitle={International Conference on Complex Networks and Their Applications},
	pages={189-199},
	year={2020},
	organization={Springer}
}

@article{jaccard-5,
	title={Identifying banking transaction descriptions via support vector machine short-text classification based on a specialized labelled corpus},
	author={Garcia-Mendez, Silvia and Fernandez-Gavilanes, Milagros and Juncal-Martinez, Jonathan and Gonz{\'a}lez-Casta{\~n}o, Francisco Javier and Seara, {\'O}scar Barba},
	journal={IEEE Access},
	volume={8},
	pages={61642-61655},
	year={2020},
	publisher={IEEE}
}

@article{jaccard-6,
	title={Plant recognition based on Jaccard distance and BOW},
	author={Wang, Zhaobin and Cui, Jing and Zhu, Ying},
	journal={Multimedia Systems},
	volume={26},
	number={5},
	pages={495-508},
	year={2020},
	publisher={Springer}
}

@article{PPM,
	title={Algorithms for graph partitioning on the planted partition model},
	author={Condon, Anne and Karp, Richard M},
	journal={Random Structures \& Algorithms},
	volume={18},
	number={2},
	pages={116-140},
	year={2001},
	publisher={Wiley Online Library}
}

@article{SBM,
	title={Stochastic blockmodels: first steps},
	author={Holland, Paul W and Laskey, Kathryn Blackmond and Leinhardt, Samuel},
	journal={Social Networks},
	volume={5},
	number={2},
	pages={109-137},
	year={1983},
	publisher={Elsevier}
}

@techreport{gc-no26,
	title={Exploring network structure, dynamics, and function using NetworkX},
	author={Hagberg, Aric and Swart, Pieter J and Schult, Daniel A},
	year={2008},
	institution={Los Alamos National Laboratory (LANL), Los Alamos, NM (United States)}
}

@inproceedings{le2020farthest,
	title={A farthest first traversal based sampling algorithm for k-clustering},
	author={Le Hoang, Nguyen and Dang, Tran Khanh and others},
	booktitle={2020 14th International Conference on Ubiquitous Information Management and Communication (IMCOM)},
	pages={1-6},
	year={2020},
	organization={IEEE}
}

@article{gonzalez1985clustering,
	author  = {Gonzalez, Teofilo F.},
	title   = {Clustering to Minimize the Maximum
	Intercluster Distance},
	journal = {Theoretical Computer Science},
	volume  = {38},
	pages   = {293-306},
	year    = {1985},
	doi     = {10.1016/0304-3975(85)90224-5}
}

@article{ParkJun2009,
	author  = {Park, Hae-Sang and Jun, Chi-Hyuck},
	title   = {A Simple and Fast Algorithm for K-Medoids Clustering},
	journal = {Expert Systems with Applications},
	volume  = {36},
	number  = {2},
	pages   = {3336-3341},
	year    = {2009},
	doi     = {10.1016/j.eswa.2008.01.039}
}

@article{shao2020binary,
	title={Binary MIMO detection via homotopy optimization and its deep adaptation},
	author={Shao, Mingjie and Ma, Wing-Kin},
	journal={IEEE Transactions on Signal Processing},
	volume={69},
	pages={781-796},
	year={2020},
	publisher={IEEE}
}

@inproceedings{k-medoids,
  title     = {A QUBO Formulation of the $k$-Medoids Problem},
  author    = {Bauckhage, Christian and Piatkowski, Nico and Sifa, Rafet and Hecker, Dirk and Wrobel, Stefan},
  booktitle = {Proceedings of the Conference on Lernen, Wissen, Daten, Analysen},
  series    = {CEUR Workshop Proceedings},
  volume    = {2454},
  pages     = {54-63},
  year      = {2019}
}

@article{miasnikof2024graph,
  title   = {Graph Clustering with Boltzmann Machines},
  author  = {Miasnikof, Pierre and Bagherbeik, Mohammad and Sheikholeslami, Ali},
  journal = {Discrete Applied Mathematics},
  volume  = {343},
  pages   = {208-223},
  year    = {2024},
  doi     = {10.1016/j.dam.2023.10.012}
}

@article{liu2024expp1,
  title   = {Extreme Point Pursuit---Part I: A Framework for Constant Modulus Optimization},
  author  = {Liu, Junbin and Liu, Ya and Ma, Wing-Kin and Shao, Mingjie and So, Anthony Man-Cho},
  journal = {IEEE Transactions on Signal Processing},
  volume  = {72},
  pages   = {4541-4556},
  year    = {2024},
  doi     = {10.1109/TSP.2024.3458008}
}

@article{liu2024expp2,
  title   = {Extreme Point Pursuit---Part II: Further Error Bound Analysis and Applications},
  author  = {Liu, Junbin and Liu, Ya and Ma, Wing-Kin and Shao, Mingjie and So, Anthony Man-Cho},
  journal = {arXiv preprint arXiv:2403.06513},
  year    = {2024}
}

@article{schaeffer2007graph,
  title   = {Graph Clustering},
  author  = {Schaeffer, Satu Elisa},
  journal = {Computer Science Review},
  volume  = {1},
  number  = {1},
  pages   = {27-64},
  year    = {2007},
  doi     = {10.1016/j.cosrev.2007.05.001}
}

@article{fortunato2010community,
  title   = {Community Detection in Graphs},
  author  = {Fortunato, Santo},
  journal = {Physics Reports},
  volume  = {486},
  number  = {3-5},
  pages   = {75-174},
  year    = {2010},
  doi     = {10.1016/j.physrep.2009.11.002}
}

@article{fortunato2016community,
  title   = {Community Detection in Networks: A User Guide},
  author  = {Fortunato, Santo and Hric, Darko},
  journal = {Physics Reports},
  volume  = {659},
  pages   = {1-44},
  year    = {2016},
  doi     = {10.1016/j.physrep.2016.09.002}
}

@article{intro-24,
	title={Finding and evaluating community structure in networks},
	author={Newman, Mark EJ and Girvan, Michelle},
	journal={Physical Review E},
	volume={69},
	number={2},
	pages={026113},
	year={2004},
	publisher={APS}
}

@inproceedings{intro-add-24-1,
	title={Integrating MCMC with modularity maximization-based methods for community detection in complex networks},
	author={Ibraheem, Samaa F and Al-Sarray, Basad},
	booktitle={Third International Conference on Emerging Trends in AI and Computational Technologies (ICONEST 2025)},
	volume={14141},
	pages={253-263},
	year={2026},
	organization={SPIE}
}

@article{fan2010linear,
  title   = {Linear and Quadratic Programming Approaches for the General Graph Partitioning Problem},
  author  = {Fan, Neng and Pardalos, Panos M.},
  journal = {Journal of Global Optimization},
  volume  = {48},
  number  = {1},
  pages   = {57-71},
  year    = {2010},
  doi     = {10.1007/s10898-009-9520-1}
}

@article{intro-70,
	title={Resolution limit in community detection},
	author={Fortunato, Santo and Barthelemy, Marc},
	journal={Proceedings of the National Academy of Sciences},
	volume={104},
	number={1},
	pages={36-41},
	year={2007},
	publisher={National Academy of Sciences}
}

@inproceedings{aref2023heuristic,
  title     = {Heuristic Modularity Maximization Algorithms for Community Detection Rarely Return an Optimal Partition or Anything Similar},
  author    = {Aref, Samin and Mostajabdaveh, Mahdi and Chheda, Hriday},
  booktitle = {Computational Science -- ICCS 2023},
  series    = {Lecture Notes in Computer Science},
  pages     = {612-626},
  publisher = {Springer},
  address   = {Cham},
  year      = {2023},
  doi       = {10.1007/978-3-031-36027-5_48}
}

@article{intro-71,
	title={Performance of modularity maximization in practical contexts},
	author={Good, Benjamin H and De Montjoye, Yves-Alexandre and Clauset, Aaron},
	journal={Physical Review E-Statistical, Nonlinear, and Soft Matter Physics},
	volume={81},
	number={4},
	pages={046106},
	year={2010},
	publisher={APS}
}

@article{intro-72,
	title={Robust optimization of graph partitioning involving interval uncertainty},
	author={Fan, Neng and Zheng, Qipeng P and Pardalos, Panos M},
	journal={Theoretical Computer Science},
	volume={447},
	pages={53-61},
	year={2012},
	publisher={Elsevier}
}

@inproceedings{gc-no34,
	title={Graph clustering via intra-cluster density maximization},
	author={Miasnikof, Pierre and Pitsoulis, Leonidas and Bonner, Anthony J and Lawryshyn, Yuri and Pardalos, Panos M},
	booktitle={International Conference on Network Analysis},
	pages={37-48},
	year={2018},
	organization={Springer}
}

@article{rand1971objective,
	author  = {Rand, William M.},
	title   = {Objective Criteria for the Evaluation of Clustering Methods},
	journal = {Journal of the American Statistical Association},
	volume  = {66},
	number  = {336},
	pages   = {846-850},
	year    = {1971},
	doi     = {10.1080/01621459.1971.10482356}
}

@article{hubert1985comparing,
	author  = {Hubert, Lawrence and Arabie, Phipps},
	title   = {Comparing Partitions},
	journal = {Journal of Classification},
	volume  = {2},
	pages   = {193-218},
	year    = {1985},
	doi     = {10.1007/BF01908075}
}

@article{steinley2004properties,
	author  = {Steinley, Douglas},
	title   = {Properties of the {Hubert--Arabie} Adjusted {Rand} Index},
	journal = {Psychological Methods},
	volume  = {9},
	number  = {3},
	pages   = {386-396},
	year    = {2004},
	doi     = {10.1037/1082-989X.9.3.386}
}

@article{no3,
	title={Community structure in social and biological networks},
	author={Girvan, Michelle and Newman, Mark EJ},
	journal={Proceedings of the National Academy of Sciences},
	volume={99},
	number={12},
	pages={7821-7826},
	year={2002},
	publisher={The National Academy of Sciences}
}

@article{zheng2026exact,
	title   = {Exact-Penalty Prox-Linear Methods for Bilevel Optimization with $\ell_1$ Lower-Level Gradient Penalty},
	author  = {Zheng, Yutong and Li, J. and Li, Qingna},
	journal = {Axioms},
	year    = {2026},
	volume  = {15},
	number  = {7},
	pages   = {512}
}
	
	\appendix
    \section{Extended-Time Results of Gurobi on SBM Instances}
    \label{app:gurobi-extended}
	
	To further investigate the effect of increasing the computational budget of the commercial solver, we provide additional results obtained by allowing Gurobi to run for up to \(500\) s on the SBM instances. Unlike the primary comparison, where Gurobi is restricted to a \(100\)-s time limit for a budget-matched evaluation, these experiments are intended only to examine whether a substantially longer optimization time can improve the obtained solutions.
	
	The reported objective values correspond to the best feasible incumbent solutions obtained by Gurobi within the \(500\)-s time limit. Since global optimality is not necessarily certified within the given time limit, these values should be interpreted as incumbent solutions rather than proven optimal solutions.

	\begin{table}[htbp]
		\centering
		\caption{Extended-time optimization results of Gurobi on SBM instances with a \(500\)-s time limit.}
		\label{tab:gurobi-500-opt}
		\setlength{\tabcolsep}{2.4pt}
		\begin{tabular*}{0.9\textwidth}{@{\extracolsep{\fill}}cccccccccc}
			\toprule
			Instance & G1\_SBM & G2\_SBM & G3\_SBM& G4\_SBM & G5\_SBM & G6\_SBM& G7\_SBM & G8\_SBM & G9\_SBM \\
			\midrule
			$f(\mathbf{x})$ & -11.4121 & -10.4020 & -9.7324 & -9.7078 & -8.9746 & -9.5735 & -9.5785 & -8.7923 & -8.9331\\
			     CPU time(s)  & 451.34 & 443.34 & 451.39 & 414.31 & 404.80 & 493.5566 & 486.25 & 452.22 & 502.53 \\
			\bottomrule
			\end{tabular*}
	\end{table}
	
   The medoids obtained from the extended Gurobi runs are further processed by the same nearest-medoid assignment procedure.
   
   \begin{table}[ht]
   	\centering
   	\caption{Clustering quality obtained from the \(500\)-s Gurobi solutions after unified assignment.}
   	\label{tab:gurobi-500-clustering}
   	\setlength{\tabcolsep}{2.8pt}
   	\begin{tabular*}{0.9\textwidth}{@{\extracolsep{\fill}}cccccccccc}
   		\toprule
   		Instance & G1\_SBM & G2\_SBM & G3\_SBM& G4\_SBM & G5\_SBM & G6\_SBM& G7\_SBM & G8\_SBM & G9\_SBM \\
   		\midrule
   		$\epsilon_{\rm intra}$ & 0.0020 & 0.0808 & 0.4630 & 0.2153 & 0.2802 & 0.3894 & 0.2794 & 0.2046 & 0.3690 \\
   		$\epsilon_{\rm inter}$ & 0.0002 & 0.0042 & 0.0232 & 0.0098 & 0.0114 & 0.0200 & 0.0141 & 0.0096 & 0.0184 \\
   	                       ARI & 0.9996 & 0.9308 & 0.5009 & 0.7861 & 0.7527 & 0.5859 & 0.7660 & 0.7732 & 0.5250 \\
   	             ${\rm T_{total}}$(s) & 451.48 & 443.43 & 452.03 & 415.14 & 404.88 & 494.83 & 487.63 & 452.40 & 502.98 \\
   		\bottomrule
   	\end{tabular*}
   \end{table}
   
	The extended-time results show that increasing the Gurobi time limit from \(100\) s to \(500\) s can improve the incumbent objective values on some instances. However, the additional computational effort does not substantially change the final clustering comparison. The clustering quality obtained from the extended Gurobi runs remains inferior to or comparable with that of TSCM, while requiring a significantly larger computational budget.

\end{document}